\documentclass[final,onefignum,oneeqnum,onetabnum]{siamart171218}

\usepackage{url}

\usepackage{amsmath}
\usepackage{amsfonts}
\usepackage{amssymb}

\newcommand\N{\mathbb{N}}
\newcommand\R{\mathbb{R}}

\newcommand\abs[1]{|#1|}

\newcommand{\inner}[2]{\langle#1,#2\rangle}

\newcommand\set[1]{\{#1\}}
\newcommand\Set[1]{\left\{#1\right\}}

 \newtheorem{remark}{Remark}[section]

\newcommand{\be}{\begin{equation}}
\newcommand{\ee}{\end{equation}}

\def\dom{{\rm dom}}

 \renewcommand{\theequation}{\thesection.\arabic{equation}}

\usepackage{amsfonts}
\usepackage{graphicx}
\usepackage{epstopdf}
\usepackage{algorithmic}
\usepackage{booktabs}
\ifpdf
  \DeclareGraphicsExtensions{.eps,.pdf,.png,.jpg}
\else
  \DeclareGraphicsExtensions{.eps}
\fi

\newcommand{\TheTitle}{Inexact cuts in SDDP applied to multistage stochastic nondifferentiable problems}
\newcommand{\TheAuthors}{}

\author{
  Vincent Guigues\thanks{School  of Applied Mathematics, Funda\c{c}\~ao Getulio Vargas,
190 Praia de Botafogo, Rio de Janeiro, Brazil,
    (\email{vincent.guigues@fgv.br}). Research of this author was partially supported
    by CNPq grants 401371/2014-0, 204872/2018-9 and 311289/2016-9.}
  \and
     Renato Monteiro\thanks{
     School of Industrial and Systems Engineering,
Georgia Institute of Technology, Atlanta, GA 30332-0205, USA,
    (\email{renato.monteiro@isye.gatech}).
    Research of this author was partly supported by CNPq grant 401371/2014-0.}
    \and
      Benar Svaiter\thanks{
     IMPA, Estrada Dona Castorina, 110, Jardim Bot\^anico, Rio de Janeiro, Brazil.
    (\email{benar@impa.br}).
    Research of this author was partly supported by CNPq grant 401371/2014-0}
}

\title{{\TheTitle}}

\usepackage{amsopn}

\ifpdf
\hypersetup{
  pdftitle={\TheTitle},
  pdfauthor={\TheAuthors}
}
\fi

\begin{document}

\maketitle

\begin{abstract} In \cite{guigues2016isddp}, an Inexact variant of Stochastic Dual Dynamic Programming (SDDP) called
ISDDP was introduced which uses approximate (instead of exact with SDDP) primal dual solutions of the problems solved
in the forward and backward passes of the method. That variant of SDDP was studied in \cite{guigues2016isddp}
for linear and for differentiable nonlinear Multistage Stochastic Programs (MSPs). In this paper, we
extend ISDDP to nondifferentiable MSPs. We first provide formulas for inexact cuts for value functions
of convex nondifferentiable optimization problems. We then combine these cuts with SDDP to describe ISDDP
for nondifferentiable MSPs and analyze the convergence of the method. More precisely, 
for a problem with $T$ stages, we show that
for errors bounded from above by $\varepsilon$, the limit superior and limit inferior of sequences of upper and lower bounds on the optimal
value of the problem are at most at distance $3 \varepsilon  T$ to the optimal value
and that for asymptotically vanishing errors ISDDP converges to an optimal policy.
Finally, we present the results
of encouraging numerical experiments
on a multistage nondifferentiable
stochastic convex program solved
using exact SDDP and the proposed
inexact variant of SDDP.
\end{abstract}



\section{Introduction}

Multistage stochastic programs (MSPs) offer a framework to model many real-life applications but are challenging
to solve, see \cite{shadenrbook} for a thorough review on MSPs.

A possible approach to approximately solve such problems is to restrict the policies
to be decision rules belonging to specific classes of parametric functions, see for instance 
\cite{kuhn11} and references therein. In this situation, most studies have focused on classes of problems and of decision rules allowing for a reformulation of the problem (either tight or with controlled accuracy) as a tractable optimization problem, i.e., a well structured convex optimization problem.
This strategy has also been used in the context of Robust Optimization where 
uncertain parameters are assumed to belong to convex, nonempty, compact sets (see \cite{nembook}
for a thorough presentation of Robust Optimization) for instance in \cite{nem}.

Another approach to solve MSPs formulated using Dynamic Programming equations
is to approximate the recourse functions. Two important classes of such methods
are Approximate Dynamic Programming \cite{powellbook} and Stochastic Dual Dynamic Programming (SDDP)
\cite{pereira} which is a sampling-based extension of the Nested Decomposition method \cite{birgemulti},
closely related to Stochastic Decomposition \cite{hise96}.

Several variants of SDDP have been proposed such as CUPPS \cite{chenpowell99}, 
ReSa \cite{resa}, the Abridged Nested Decomposition \cite{birgedono}, MIDAS
\cite{bonnansphilp19} for monotonic Bellman functions, or risk-averse variants \cite{guiguesrom10}, \cite{shapsddp}, \cite{guiguescoap2013}, \cite{kozmikmorton}.
For convergence analysis of the method and variants see \cite{philpot},\cite{lecphilgirar12},\cite{guiguessiopt2016}, \cite{bandarraguigues}.
We also refer to \cite{shap192} which explains how to take advantage of the stationarity of the
underlying stochastic processes to solve MSPs with SDDP and to \cite{guigues2016isddp}, \cite{shap191}
for variants which can accelerate the convergence of SDDP. 
In particular, in \cite{guigues2016isddp},  an Inexact variant of SDDP called ISDDP was introduced
which allows us to solve approximately the optimization subproblems of the forward and backward passes
of SDDP and to increase the accuracy of the solutions of these subproblems along the iterations
of the method. 
ISDDP can be seen as an extension  
to multistage and both linear and nonlinear problems of  \cite{philpzakinex} 
where inexact cuts were combined with
Benders Decomposition \cite{benderscut} to solve two-stage stochastic linear programs.
An inexact Stochastic Dynamic Cutting Plane (another variant of SDDP solving approximately
the subproblems along the iterations of the method) was also introduced 
in \cite{guiguesmonteiro2019} to solve MSPs. For all these inexact variants,
convergence can be shown for vanishing noises and numerical experiments in \cite{philpzakinex}, \cite{guigues2016isddp}  
have shown that convergence
can be achieved quicker with these inexact variants.

The motivation for introducing inexact cuts obtained from the approximate primal-dual solutions
of the convex nonlinear subproblems generated during the course of SDDP is due to the following reasons:
\begin{itemize}
\item[(i)]  a convex nonlinear subproblem can take a significant amount of time or may even be impossible to be
solved
to high accuracy;
\item[(ii)] it is advantageous from a practical point of view to solve the initial subproblems generated by SDDP
with much less accuracy than the ones generated
during its late stages; in fact, the implementation presented in
\cite{guigues2016isddp} shows that an inexact SDDP variant based on this idea
outperforms exact SDDP 
on several instances of a portfolio problem (see also the numerical experiments in Section \ref{sec:numexp} below).

\end{itemize}
 
In this paper, we extend the results
of \cite{guigues2016isddp} to the nondifferentiable case, proposing
and studying Inexact SDDP for 
possibly nondifferentiable multistage
stochastic convex programs.
More precisely, the contributions
of this paper are given below.\\

\par {\textbf{Contributions.}}\\

\par {\textbf{A. Deriving
formulas for inexact cuts for value functions
of possibly nondifferentiable optimization problems.}} An important tool in the development of inexact variants of SDDP is the 
computation of inexact cuts for value functions of optimization problems, i.e.,
affine lower bounding functions for the value function on the basis of approximate primal-dual
solutions. This task can be easily achieved for value functions of linear programs, see for instance
Proposition 2.1 in \cite{guigues2016isddp}. For nonlinear differentiable problems, the derivation
of inexact cuts is given in Propositions 2.2 and 2.3 in \cite{guigues2016isddp} and Proposition 3.8 in \cite{guiguesinexactsmd}. 
However, this task is more complicated for nondifferentiable optimization problems.

\par We extend these results developping tools to compute
inexact cuts for value functions of nondifferentiable optimization problems. 
Mathematically, the problem
can be stated as follows. Let $\mathcal{Q}: X \rightarrow \mathbb{R}$ be the value function given by
\begin{equation} \label{vfunction1}
\mathcal{Q}(x)=\left\{
\begin{array}{l}
\min_{
y \in \mathbb{R}^m} \;f(y, x)\\
y \in Y, Ay + B x =b, g_i(y, x ) \leq 0, i=1,\ldots,p,
\end{array}
\right.
\end{equation}
where $X \subseteq \mathbb{R}^n, Y \subseteq \mathbb{R}^m$ and where \\
\par (H0) $X$ and $Y$ are convex, closed, and nonempty sets and $f, g_i: Y \small{\times} X \rightarrow 
]-\infty,+\infty]$ are proper, lower semicontinuous, convex, and possibly nondifferentiable.\\

Due to (H0) value function $\mathcal{Q}$ is convex and if $\bar x \in \mbox{ri}(\mbox{dom}(\mathcal{Q}))$
then $\mathcal{Q}$ is subdifferentiable at $\bar x$ and there exists a cut (a lower bounding affine function)
for $\mathcal{Q}$ at $\bar x$ which coincides with $\mathcal{Q}$ at $\bar x$.
More generally, under some assumptions, the characterization of the subdifferential of $\mathcal{Q}$ at $\bar x \in X$ was given in
\cite[Lemma 2.1]{guiguessiopt2016} and formulas for affine lower bounding functions for 
$\mathcal{Q}$ were derived in
\cite[Proposition 3.2]{guiguesinexactsmd}
on the basis of optimal primal-dual solutions to \eqref{vfunction1}.
When only approximate primal-dual solutions are available, we can only compute inexact cuts which are still lower bounding functions for the value function
but which do not coincide with this function at the point $\bar x$
used to compute the cut.
Formulas for computing inexact cuts on the basis of approximate primal-dual solutions to \eqref{vfunction1} were derived in
\cite{guigues2016isddp, guiguesinexactsmd} when functions $f, g_i$ are differentiable.
In this paper, we extend in
Sections \ref{sec:icut1}, \ref{sec:icut2} this analysis considering possibly nondifferentiable functions $f, g_i$.

\par {\textbf{A.1).}} More precisely, in Section \ref{sec:icut1} we derive inexact cuts using a reformulation
of the problem that adds some variables and constraints. Such copies of (state)
variables have been
used to derive cuts in several
publications, for instance 
\cite{sddip}. The novelty of the
cuts we derive comes from the
fact that they are built on the basis
of approximate primal-dual solutions
and we provide the level of inexactness
of the cuts, see Proposition \ref{icutsec2}
and Corollary \ref{intcor}. 
In particular, Corollary \ref{intcor}
provides cuts easier to compute
than the inexact cuts from
\cite{guigues2016isddp} and
easy to interpret. Indeed, 
while the computation of the cuts
from \cite{guigues2016isddp}
requires solving an additional
optimization problem, Corollary
\ref{intcor}  provides an analytic formula
for the inexact cuts with
the slope being simply an 
approximate dual solution, the
intercept being the dual problem
approximate optimal value, and
the level of inexactness being the 
difference
between the approximate primal and dual
optimal values. 
For convex problems,
such copy of state variables 
is not needed to compute
exact cuts (on the basis of exact
primal-dual solutions), see
\cite[Lemma 2.1]{guiguessiopt2016},
but it offers a simple way
to derive cuts in the inexact case.
\par {\textbf{A.2).}} Section \ref{sec:icut2} provides formulas for inexact cuts when the objective $f$ has
a saddle point representation.
The advantage of these cuts, compared
to the cuts derived in
Section \ref{sec:icut1},
is that they are computed without
adding additional variables and constraints.

\par {\textbf{B. Comparison with the cuts
from \cite{guigues2016isddp} in the differentiable case.}} In the case when $f$ and $g_i$ are differentiable,
we compare in Section \ref{sec:comparison} the formulas for inexact cuts from \cite{guigues2016isddp} and the formulas 
from Section \ref{sec:icut1}. 
In particular, on the basis
of characterizations of approximate
$\varepsilon$-optimal
primal-dual solutions, we provide upper
bounds on the level of inexactness
of the cuts.

\par {\textbf{C. Inexact cuts in SDDP
for nondifferentiable problems.}}
In Section \ref{sec:isddp}, we describe ISDDP for possibly nondifferentiable MSPs
combining the framework of SDDP with the inexact cuts derived in Sections \ref{sec:icut1} and \ref{sec:icut2}.

\par {\textbf{D. Convergence of Inexact SDDP
for nondifferentiable problems.}}
In Section \ref{sec:isddp}, we also study the convergence of ISDDP. 
A useful tool for the convergence analysis of SDDP and ISDDP is Lemma 5.2 in \cite{lecphilgirar12} for vanishing errors
and Lemma 4.1 in \cite{guigues2016isddp} for bounded errors. 
We provide different proofs of these lemmas with slightly different assumptions (see the corresponding
Lemmas \ref{technicallemma} and \ref{limsuptechlemma}) and derive a stronger conclusion.
More precisely, one of our assumptions is stronger (the continuity of $f$ [which is satisfied when the lemmas are
applied to study the convergence of ISDDP]) and two are weaker.
We  show the almost sure uniform convergence of the approximate Bellman
functions generated by ISDDP to a continuous function which coincides with the true Bellman functions at all accumulation points
of the sequences of trial points. Interestingly, as for ISDDP applied to linear
programs studied in \cite{guigues2016isddp}, we show that 
for a problem with $T$ stages and errors bounded from above by $\varepsilon$, the limit superior and limit inferior of sequences of upper and lower bounds on the optimal
value of the problem are at most at distance $3 \varepsilon  T$ to the optimal value.
Finally, similarly to ISDDP for nonlinear differentiable programs developped in \cite{guigues2016isddp}, we show the convergence of ISDDP to an optimal policy for vanishing noises.

\par {\textbf{E. Numerical experiments.}}
We consider 2 instances of a nondifferentiable
multistage stochastic program
and solve them using both
exact and inexact variants of
SDDP (the one proposed in
\cite{guigues2016isddp}
and Inexact SDDP given in this paper).
We also consider a solution
method called {\tt{MSDDP}}
 mixing StoDCuP from 
\cite{guiguesmonteiro2019}
and Inexact SDDP. 
On these experiments, 
the inexact variants of {\tt{MSDDP}}
and of SDDP developped in this paper
converge quicker than (exact) SDDP.
\if{

\section{Inexact cuts for value functions of convex nondifferentiable problems
with the argument in the objective function only}\label{sec:inexccutuncons}

\begin{definition}[$\varepsilon$-inexact cut.] Let $\mathcal{Q}:X \rightarrow \mathbb{R}$ be a convex function
with $X$ convex, $X \subset \emph{ri}(\emph{dom}(\mathcal{Q}))$, and let $\varepsilon \geq 0$. We say that $\mathcal{C}: X \rightarrow  \mathbb{R}$ is an $\varepsilon$-inexact cut
for $\mathcal{Q}$ at $\bar x \in X$ if $\mathcal{C}$ is an affine function satisfying
$\mathcal{Q}( x ) \geq \mathcal{C}( x)$ for all $x \in X$
and $\mathcal{Q}( \bar x ) - \mathcal{C}( \bar x ) \leq \varepsilon$.
\end{definition}
\begin{remark} A 0-inexact cut for $\mathcal{Q}$ at $\bar x$, i.e., an $\varepsilon$-inexact cut at $\bar x$ with $\varepsilon=0$ will be called an exact cut for
$\mathcal{Q}$ at $\bar x$.
\end{remark}

In this section, we consider the case when \eqref{vfunction1} has 
constraints that do not depend on $x$. The results obtained in this case
will be useful to consider problems of form \eqref{vfunction1}.

Let $f: \mathbb{R}^m \small{\times} \mathbb{R}^n \rightarrow (-\infty,+\infty]$ be a proper closed convex
(not necessarily differentiable) function and consider the parametrized minimization problem
\begin{equation}\label{defvalueqsimple}
\mathcal{Q}(x) = \inf_{y \in \mathbb{R}^m} f_x(y) := f(y,x)
\end{equation}
and assume that $\dom(Q) \ne \emptyset$. We start with a few elementary observations that will be used in the sequel.

\begin{proposition}\label{propvalueffirst} Consider value function $\mathcal{Q}$ given by \eqref{defvalueqsimple}.
Then:
\begin{itemize}
 \item[(i)] the conjugate of $\mathcal{Q}$ is given by 
 $\mathcal{Q}^*(\lambda)=f^*(0,\lambda)$ for all $\lambda \in \mbox{dom}(\mathcal{Q})$.
 \item[(ii)] For every $x \in \mbox{dom}(\mathcal{Q})$, we have
 $\overline{\emph{co}}\; \mathcal{Q} (x) = \displaystyle \sup_{\lambda \in \mathbb{R}^n } \langle \lambda, x \rangle - f^*(0,\lambda)$.
\end{itemize}
\end{proposition}
\begin{proof} (i) We compute
\begin{eqnarray*}
\mathcal{Q}^*(\lambda) &  = & \displaystyle \sup_{x \in \mathbb{R}^n} \, \langle \lambda , x \rangle - \mathcal{Q}(x)  =  
\sup_{x \in \mathbb{R}^n } \,
 \langle \lambda , x \rangle - \inf_{y \in \mathbb{R}^m} f(y,x) \\
& = & \sup_{(y,x) \in \mathbb{R}^m \times \mathbb{R}^n } \langle \lambda , x \rangle - f(y,x)  =  f^*(0,\lambda).
\end{eqnarray*}
(ii) Knowing that $\mbox{cl }\mbox{co} \mathcal{Q}(x) =  \mathcal{Q}^{**}( x )$ (see 6.3.3, p.44 in \cite{hhlem}), we also have
\begin{eqnarray*}
\mbox{cl }\mbox{co} \mathcal{Q} (x)  =   \mathcal{Q}^{**}( x )  =  \sup_{\lambda \in \mathbb{R}^n } \, \langle \lambda , x \rangle - \mathcal{Q}^*(\lambda)  \stackrel{(i)}{=}  \sup_{\lambda \in \mathbb{R}^n} \, \langle \lambda , x \rangle -f^*(0,\lambda).
\end{eqnarray*}
\end{proof}
We define the dual function
\begin{equation}\label{dualfunctioninit}
\theta_x(\lambda) =\langle \lambda , x \rangle - f^*(0,\lambda)
\end{equation}
and the dual problem
\begin{equation}\label{dualgener}
{\underline{\mathcal{Q}}}(x) = \sup_{\lambda \in \mathbb{R}^n} \theta_x( \lambda ).
\end{equation}

\begin{remark} Problem \eqref{defvalueqsimple} can be re-written 
\begin{equation}\label{defvalueqsimple2}
\mathcal{Q}(x) = \inf_{y \in \mathbb{R}^m, z \in \mathbb{R}^n} \{ f(y,z) : z=x \}.
\end{equation}
Function $\theta_x(\lambda)$ is the Lagrangian dual function
associated with the Lagrangian function
$$
L_x(y,z;\lambda)=f(y,z)+\langle \lambda , x - z \rangle 
$$
obtained dualizing the  constraint $z=x$ in \eqref{defvalueqsimple2}. Indeed, 
\begin{eqnarray*}
\theta_x( \lambda ) =  \langle  \lambda , x \rangle - f^{*}(0,\lambda ) = 
\langle \lambda , x \rangle - \sup_{(y,z) \in \mathbb{R}^m \times \mathbb{R}^n} \langle \lambda , z \rangle -f(y,z) =
\inf_{(y,z) \in \mathbb{R}^m \times \mathbb{R}^n} L_x(y,z;\lambda).
\end{eqnarray*}
\end{remark}

\begin{proposition}[Weak duality] For every $y \in \mathbb{R}^m$ and $ \lambda, x \in \mathbb{R}^n$, we have
$$
f_x(y) \geq \mathcal{Q}(x) \geq {\underline{\mathcal{Q}}}(x) \geq \theta_x(\lambda).
$$
\end{proposition}
\begin{proof}
The first and third inequalities are immediate while the second inequality comes from the fact that
$$
\mathcal{Q}(x) \geq \mathcal{Q}^{**}(x)  = \sup_{\lambda \in \mathbb{R}^n} \theta_x( \lambda ) =  {\underline{\mathcal{Q}}}(x),
$$
where for the second equality we have used Proposition \ref{propvalueffirst} and the definition of 
$\theta_x$.
\hfill
\end{proof}

\begin{definition}[$\varepsilon$-optimal primal-dual solution.]
Let $\varepsilon \geq 0$.
$(\hat y, \hat \lambda)$ is an
 $\varepsilon$-optimal primal-dual solution for \eqref{defvalueqsimple} if
 $f(\hat y, x ) - \theta_x( \hat \lambda ) \leq \varepsilon$.
\end{definition}
\begin{definition}[$\varepsilon$-optimal dual solution.]
Let $\varepsilon \geq 0$.
$\hat \lambda$ is a   n
 $\varepsilon$-optimal dual solution for \eqref{defvalueqsimple} if
 ${\underline{\mathcal{Q}}}(x) - \varepsilon \leq \theta_x( \hat \lambda )$.
\end{definition}
\begin{definition} For $x \in \mathbb{R}^n$,  the dual gap for \eqref{defvalueqsimple} is given by
$$
{\tt{Gap}}(x)=\mathcal{Q}(x)-{\underline{\mathcal{Q}}}(x) \geq 0.
$$
\end{definition}

We obtain the following characterization of the $\varepsilon$-subdifferential of $\mathcal{Q}$:

\begin{proposition}\label{epssubsimplecase} Let $\mathcal{Q}$ be the value function given by  \eqref{defvalueqsimple}
and let $\bar x \in \mathbb{R}^n$.
Then
$$
\hat \lambda \in \partial_{\varepsilon} \mathcal{Q}( \bar x )
\Longleftrightarrow {\tt{Gap}}(\bar x) + {\underline{\mathcal{Q}}}(\bar x) - \theta_{\bar x}(\hat \lambda  ) \leq \varepsilon.
$$
\end{proposition}
\begin{proof}
We have
\begin{eqnarray*}\
\hat \lambda \in \partial_{\varepsilon} \mathcal{Q}( \bar x )
&\Longleftrightarrow &
\mathcal{Q}(x) \geq \mathcal{Q}(\bar x) + \langle \hat \lambda , x - \bar x \rangle - \varepsilon,\; \forall \;x \in \mathbb{R}^n,\\
&\Longleftrightarrow &
\varepsilon - \mathcal{Q}(\bar x) + \langle \hat \lambda , \bar x \rangle
\geq \sup_{x \in \mathbb{R}^n} \langle \hat \lambda , x \rangle - \mathcal{Q}(x)\\
&\Longleftrightarrow &
\varepsilon - \mathcal{Q}(\bar x) + \langle \hat \lambda , \bar x \rangle
\geq \mathcal{Q}^{*}(\hat \lambda )\\
&\Longleftrightarrow &
{\tt{Gap}}(\bar x) + {\underline{\mathcal{Q}}}(\bar x) - \theta_{\bar x}(\hat \lambda  ) \leq \varepsilon.
\end{eqnarray*}
\hfill
\end{proof}
We obtain the following immediate corollary of Proposition \ref{epssubsimplecase}:

\begin{corollary}\label{firstimmcor} Let $\bar x \in \mathbb{R}^n$ be given and assume that
$\mathcal{Q}(\bar x) = {\underline{\mathcal{Q}}}(\bar x)$.
 Then, the following statements are equivalent:
 \begin{itemize}
 \item[(a)] $\hat \lambda \in \partial_{\varepsilon} \mathcal{Q}( \bar x )$;
 \item[(b)] $\hat \lambda$ is a $\varepsilon$-optimal dual solution;
 \item[(c)] $\bar x \in \partial_{\varepsilon}(f^*(0,\cdot))(\hat \lambda)$.
 \end{itemize}
\end{corollary}

\begin{proposition} Let $\bar x \in \mathbb{R}^n$ be given and assume that
$\mathcal{Q}(\bar x ) = {\underline{\mathcal{Q}}}(\bar x)$. Then
\begin{eqnarray*}
\hat \lambda \in \partial_{\varepsilon} \mathcal{Q}( \bar x )
&\Longleftrightarrow &  \exists \; \bar y \in \mathbb{R}^n \,|\,\left(
\begin{array}{l}
0\\
\hat \lambda
\end{array}
\right)
\in \partial_{\varepsilon +f(\bar y , \bar x ) - \mathcal{Q}(\bar x)} f(\bar y, \bar x ), \\
&\Longleftrightarrow &\left(
\begin{array}{l}
0\\
\hat \lambda
\end{array}
\right)
\in \partial_{\varepsilon +f(\bar y , \bar x ) - \mathcal{Q}(\bar x)} f(\bar y, \bar x )
\;\emph{ for all }\bar y \in \mathbb{R}^n.  
\end{eqnarray*}
\end{proposition}
\begin{proof}We have 
\begin{eqnarray*}
\hat \lambda \in \partial_{\varepsilon} \mathcal{Q}( \bar x )
& \stackrel{Cor. \ref{firstimmcor}}{\Longleftrightarrow} & \theta_x( \hat \lambda ) \geq {\underline{\mathcal{Q}}}(\bar x) - \varepsilon  \\
&\Longleftrightarrow & \langle \hat \lambda, \bar x \rangle -f^*(0, \hat \lambda ) \geq \mathcal{Q}( \bar x ) - \varepsilon \\
&\Longleftrightarrow & \langle \hat \lambda, \bar x \rangle \geq \langle \hat \lambda , x \rangle -f(y,x) + \mathcal{Q}( \bar x ) - \varepsilon \mbox{ for all }x, y,\\
&\Longleftrightarrow & f(y,x ) \geq f(\bar y , \bar x ) + \langle \hat \lambda, x-\bar x \rangle -[\varepsilon + f(\bar y , \bar x) - \mathcal{Q}( \bar x )] \mbox{ for all }x, y,\bar y,\\
&\Longleftrightarrow & \exists \bar y \;:\;  f(y,x ) \geq f(\bar y , \bar x ) + \langle \hat \lambda, x-\bar x \rangle -[\varepsilon + f(\bar y , \bar x) - \mathcal{Q}( \bar x )] \mbox{ for all }x, y,\\
&\Longleftrightarrow &  \exists \; \bar y \in \mathbb{R}^n \,|\,\left(
\begin{array}{l}
0\\
\hat \lambda
\end{array}
\right)
\in \partial_{\varepsilon +f(\bar y , \bar x ) - \mathcal{Q}(\bar x)} f(\bar y, \bar x ), \\
&\Longleftrightarrow &\left(
\begin{array}{l}
0\\
\hat \lambda
\end{array}
\right)
\in \partial_{\varepsilon +f(\bar y , \bar x ) - \mathcal{Q}(\bar x)} f(\bar y, \bar x )
\;\mbox{ for all }\bar y \in \mathbb{R}^n.
\end{eqnarray*}
\end{proof}

\begin{corollary} Assume that  $(\bar y, \bar x) \in \mathbb{R}^m \times \mathbb{R}^n$ is such that
$f(\bar y, \bar x)  = \mathcal{Q}(\bar x ) = {\underline{\mathcal{Q}}}(\bar x)$.
Then
\begin{eqnarray*}
\bar \lambda \in \partial_{\varepsilon} \mathcal{Q}( \bar x )
&\Longleftrightarrow &  \left(
\begin{array}{l}
0\\
\bar \lambda
\end{array}
\right)
\in \partial_{\varepsilon} f(\bar y, \bar x ).
\end{eqnarray*}
\end{corollary}

}\fi

\section{Inexact cuts for value functions of convex optimization problems} \label{sec:icut1}

In the sequel, the usual scalar product in $\mathbb{R}^n$ is denoted by 
$\langle x, y\rangle = x^{\top} y$ for $x, y \in \mathbb{R}^n$.
The corresponding norm is $\|x\|= \|x\|_2=\sqrt{\langle x, x \rangle}$.

The objective of this section is to compute inexact cuts with controlled accuracy $\varepsilon$ for value functions
$\mathcal{Q}$ of form \eqref{vfunction1} on the basis of approximate primal-dual solutions to \eqref{vfunction1}
solved for a given $x=\bar x$. We will call these cuts $\varepsilon$-inexact cuts at $\bar x$: 
\begin{definition}[$\varepsilon$-inexact cut.] Let $\mathcal{Q}:X \rightarrow \mathbb{R}$ be a convex function
with $X$ convex, $X \subset \emph{ri}(\emph{dom}(\mathcal{Q}))$, and let $\varepsilon \geq 0$. We say that $\mathcal{C}: X \rightarrow  \mathbb{R}$ is an $\varepsilon$-inexact cut
for $\mathcal{Q}$ at $\bar x \in X$ if $\mathcal{C}$ is an affine function satisfying
$\mathcal{Q}( x ) \geq \mathcal{C}( x)$ for all $x \in X$
and $\mathcal{Q}( \bar x ) - \mathcal{C}( \bar x ) \leq \varepsilon$.
\end{definition}
\begin{remark} A 0-inexact cut for $\mathcal{Q}$ at $\bar x$, i.e., an $\varepsilon$-inexact cut at $\bar x$ with $\varepsilon=0$ will be called an exact cut for
$\mathcal{Q}$ at $\bar x$.
\end{remark}

\subsection{Affine functions of the argument in the constraints}

We start computing inexact cuts for particular value functions $\mathcal{Q}$
where the argument of this function only appears in the constraints through affine functions
of this argument. The study of this case will help us discuss the general case of a value function
of form \eqref{vfunction1} considered in the next Section \ref{sec:gencut}.

More precisely, we consider value functions $\mathcal{Q}$ of form: 
\begin{equation}\label{pbinitprimal1}
\mathcal{Q}(x) = 
\left\{ 
\begin{array}{l}
\displaystyle \min_{y \in \mathbb{R}^m} \;f(y) \\
g(y) \leq Cx, \\
A y + Bx = b,\\
y \in Y,
\end{array}
\right.
\end{equation} 
along with the corresponding dual problem given by
\begin{equation}\label{pbinitdual1}
\left\{ 
\begin{array}{l}
\displaystyle \max_{\lambda, \mu} \;\theta_{x}(\lambda, \mu )\\
\mu \geq 0, \lambda,
\end{array}
\right.
\end{equation}
where dual function $\theta_{x}(\lambda, \mu )$ is given by
\begin{equation}\label{defdual1x}
\theta_{x}(\lambda, \mu ) =\min \{L_x(y,\lambda,\mu) :  y \in Y  \}
\end{equation}
for the Lagrangian 
$$
L_x(y,\lambda,\mu) = 
f(y) + \langle \lambda , A y + B x - b  \rangle + \langle \mu   , g(y) -Cx  \rangle.
$$
Proposition \ref{propcut1} provides a formula for computing inexact cuts for value function 
$\mathcal{Q}$ given by \eqref{pbinitprimal1}:
\begin{proposition}\label{propcut1} Assume that 
$f:\mathbb{R}^m \rightarrow ]-\infty,+\infty]$ and component functions 
$g_i:\mathbb{R}^m \rightarrow ]-\infty,+\infty],i=1,\ldots,p$, of $g$
are proper, convex, and lower semicontinuous. Assume that $\hat y$ is an 
$\varepsilon_P$-optimal feasible solution of problem \eqref{pbinitprimal1}
for $x=\bar x$ and that $(\hat \lambda , \hat \mu)$ is an $\varepsilon_D$-optimal 
feasible solution of the corresponding dual problem \eqref{pbinitdual1} for $x=\bar x$.
Assume that $f$ is finite on $\{y \in Y: Ay + B{\bar x}=b, g(y) \leq C{\bar x} \}$
and that Slater constraint qualification holds for \eqref{pbinitprimal1}
written for $x=\bar x$, i.e., there is $y_{\bar x} \in \mbox{ri}(Y)$,
such that $A y_{\bar x} + B \bar x = b$, $g( y_{\bar x} ) < C{\bar x}$.
Then 
$$
\mathcal{C}(x) = f(\hat y) - (\varepsilon_P + \varepsilon_D) + \langle B^\top \hat \lambda - C^\top \hat \mu , x - \bar x  \rangle 
$$
is an $(\varepsilon_P + \varepsilon_D)$-inexact cut for $\mathcal{Q}$ at $\bar x$.
\end{proposition}
\begin{proof}
By definition of $\hat y$, we get 
\begin{equation}\label{yepsprimal}
f( \hat y  ) \leq \mathcal{Q}( \bar x ) + \varepsilon_P.
\end{equation}
The assumptions of the Convex Duality theorem are satisfied 
for problem \eqref{pbinitprimal1} 
and its dual \eqref{pbinitdual1}, both written for $x = \bar x$.
Therefore the optimal value of dual problem 
\eqref{pbinitdual1} 
written for $x=\bar x$ is the optimal value $\mathcal{Q}( \bar x )$ of the corresponding
primal problem. Using the definition of $\hat \lambda, \hat \mu$, it follows that 
\begin{equation}\label{lambdaepsdual}
\theta_{{\bar x}}(\hat \lambda, \hat \mu )
\geq \mathcal{Q}( \bar x ) - \varepsilon_D.
\end{equation}
Next,
$$
\begin{array}{lcl}
\mathcal{Q}(x) &  \geq  &  \theta_{x}(\hat \lambda, \hat \mu ) \mbox{ by weak duality and feasibility of }\hat \mu, \hat \lambda,\\
&  =   & \min \{L_x(y,\hat \lambda,\hat \mu) :  y \in Y  \},\\
&   =   &  \langle \hat \lambda , B(x-\bar x)\rangle 
+ \langle \hat \mu , -C(x-\bar x)\rangle  + \min \{L_{\bar x}(y,\hat \lambda,\hat \mu) :  y \in Y  \}, \\
&   =   & \langle B^\top \hat \lambda - C^\top \hat \mu , x - \bar x  \rangle + \theta_{\bar x}(\hat \lambda , \hat \mu),\\
& \stackrel{\eqref{lambdaepsdual}}{\geq}  & \langle B^\top \hat \lambda - C^\top \hat \mu , x - \bar x  \rangle + \mathcal{Q}(\bar x ) - \varepsilon_D,\\
&  \stackrel{\eqref{yepsprimal}}{\geq}  &  \mathcal{C}(x) :=\langle B^\top \hat \lambda - C^\top \hat \mu , x - \bar x  \rangle + f( \hat y )- \varepsilon_P  - \varepsilon_D.
\end{array}
$$
Moreover, since $f(\hat y) \geq \mathcal{Q}(\bar x)$, we get
$$
\mathcal{Q}(\bar x) - \mathcal{C}(\bar x) =\varepsilon_P  + \varepsilon_D + \mathcal{Q}(\bar x) - f(\hat y) \leq 
\varepsilon_P  + \varepsilon_D,
$$
and we have shown that $\mathcal{C}$ is an $(\varepsilon_P  + \varepsilon_D)$-inexact cut for $\mathcal{Q}$ at $\bar x$.
\end{proof}

\begin{remark} The proof of Proposition \ref{propcut1}
also shows that if $\theta_{\bar x}(\hat \lambda , \hat \mu)$
can be computed exactly (i.e., if optimization problem \eqref{defdual1x} written
for $x=\bar x, \lambda = \hat \lambda, \mu=\hat \mu$ is solved 
to optimality) then
$\mathcal{C}(x) =  \theta_{\bar x}(\hat \lambda , \hat \mu) + \langle B^\top \hat \lambda - C^\top \hat \mu , x - \bar x  \rangle 
$ is an $\varepsilon_D$-inexact cut for $\mathcal{Q}$ at $\bar x$.
\end{remark}

\subsection{General value functions} \label{sec:gencut}

We now consider general value functions of form
\begin{equation}\label{pbinitprimal2}
\mathcal{Q}(x) = 
\left\{ 
\begin{array}{l}
\displaystyle \min_{y \in \mathbb{R}^m} \;f(y,x) \\
g(y,x) \leq 0, \\
A y + Bx = b,\\
y \in Y.
\end{array}
\right.
\end{equation}

Analyzing the proof of Proposition \ref{propcut1} dedicated to the special
case of value functions of form \eqref{pbinitprimal1}, we observe that 
the linearity  in $x$ of Lagrangian function $L$ was crucial to derive our formula
for inexact cuts. The Lagrangian obtained dualizing coupling constraints in problem \eqref{pbinitprimal2}
does not satisfy this property anymore. However, we can reformulate equivalently the problem in such a way
that the Lagrangian of the reformulated problem satisfies this property.
This reformulation is obtained adding variable $z \in \mathbb{R}^n$ together with the constraint
$z=x$.  We obtain the equivalent representation of problem \eqref{pbinitprimal2} under the form
\begin{equation}\label{pbinitprimal3}
\mathcal{Q}(x) = 
\left\{ 
\begin{array}{l}
\displaystyle \min_{y \in \mathbb{R}^m, z \in \mathbb{R}^n} \;f(y,z) \\
g(y,z) \leq 0, \\
A y + B z = b,\\
y \in Y,\\
z=x.
\end{array}
\right.
\end{equation}
The use of the copy $z=x$ of state variables
to derive cuts in the context of SDDP has been used
in several publications, for instance \cite{sddip, sddip2}.
This copy of state variables adds variables and
constraints and is not necessary
for convex problems, even for general value functions
\eqref{vfunction1} having nonlinear coupling constraints,
see  Lemma 2.1 in \cite{guiguessiopt2016} for an analytic
formula for the corresponding exact cuts.
However, the use of copy of state variables offers
a simple way to derive inexact cuts in the convex case,
see the corresponding Proposition \ref{icutsec2} and Corollary 
\ref{intcor} as well
as the more complicated computations of Section \ref{sec:icut2} that
do not use these copies of variables but use a saddle point
representation of the objective.
Denoting by $S$ the set 
\begin{equation}\label{defS}
S=\{(y,z) \in \mathbb{R}^m \small{\times} \mathbb{R}^n : g(y,z) \leq 0, A y + B z = b, y \in Y\},
\end{equation}
and dualizing the coupling constraint $z=x$ in problem \eqref{pbinitprimal3},
we obtain the dual problem given by
\begin{equation}\label{pbinitdual3}
\left\{ 
\begin{array}{l}
\displaystyle \max_{\lambda} \;\theta_{x}(\lambda)\\
\lambda \in \mathbb{R}^n,
\end{array}
\right.
\end{equation}
where dual function $\theta_{x}(\lambda)$ is given by
\begin{equation}\label{dualfunctionsecondcase}
\theta_{x}( \lambda ) = \min \{L_x(y,z,\lambda) :  (y,z) \in S \}
\end{equation}
now for the Lagrangian 
$$
L_x(y,z,\lambda ) = f(y,z) + \langle \lambda , x - z \rangle,
$$
which, as in the special case considered in the previous section, is a linear function of $x$.
Therefore, for every $x, \bar x \in X$, for every $(y,z) \in S$, and $\lambda$, we have 
$$
L_x(y,z,\lambda)=\langle  \lambda , x - \bar x \rangle + L_{\bar x}(y,z,\lambda) 
$$
and the optimal value
$\theta_{x}( \lambda )$ of problem
\eqref{dualfunctionsecondcase} is the sum of $\langle  \lambda , x - \bar x \rangle$ and of
$\theta_{\bar x}( \lambda )$. Observing that from Weak Duality $\theta_{x}( \lambda )$ is a lower bound
on $\mathcal{Q}(x)$, this sum is an affine function of $x$ which is a lower bounding function for $\mathcal{Q}$.
It can be bounded from below in terms of a computable affine function (which therefore is an inexact
cut for $\mathcal{Q}$ at $\bar x$) using an approximate primal-dual solution
if problem \eqref{pbinitprimal2} and its dual \eqref{pbinitdual3} written for $x=\bar x$ satisfy the Slater assumption.

The details of these computations are given in the proof of Proposition \ref{icutsec2} below 
which provides formulas for inexact cuts for value function \eqref{pbinitprimal2}. The proof of the proposition is given for 
completeness but, due to our previous observations, it
is similar to the proof of Proposition \ref{propcut1}.
\begin{proposition}\label{icutsec2}
Let Assumption (H0) hold. Assume that 
$\hat y$ is an 
$\varepsilon_P$-optimal feasible solution of problem \eqref{pbinitprimal2}
for $x=\bar x$ and that $\hat \lambda$ is an $\varepsilon_D$-optimal 
feasible solution of dual problem \eqref{pbinitdual3} 
written for $x=\bar x$.
Assume that 
$f(\cdot,\bar x)$ is finite on 
$\{y \in Y: Ay+b{\bar x}=b, g(y,{\bar x}) \leq 0\}$ and that
the following Slater constraint qualification holds for \eqref{pbinitprimal2}
written for $x=\bar x$:
\begin{equation}\label{slater}
\exists y_{\bar x} \mbox{ such that }(y_{\bar x} , \bar x) \in \mbox{ri}(S)
\end{equation}
where $S$ is given by \eqref{defS}.
Then 
$$
\mathcal{C}(x) = f(\hat y, \bar x) - (\varepsilon_P + \varepsilon_D) + 
\langle \hat \lambda , x - \bar x \rangle 
$$
is an $(\varepsilon_P + \varepsilon_D)$-inexact cut for $\mathcal{Q}$ at $\bar x$.
\end{proposition}
\begin{proof}
By definition of $\hat y$, we get 
$$
f( \hat y , \bar x ) \leq \mathcal{Q}( \bar x ) + \varepsilon_P.
$$
The assumptions of the Convex Duality theorem for dual problem \eqref{pbinitdual3} and primal
problem \eqref{pbinitprimal2} written for $x=\bar x$ are satisfied and therefore
the optimal value of dual problem 
\eqref{pbinitdual3} 
written for $x=\bar x$ is the optimal value $\mathcal{Q}( \bar x )$ of the corresponding
primal problem. Therefore, using the definition of $\hat \lambda$, we get
$$
\theta_{{\bar x}}(\hat \lambda )
\geq \mathcal{Q}( \bar x   ) - \varepsilon_D.
$$
Next,
$$
\begin{array}{lll}
\mathcal{Q}(x) &  \geq  &  \theta_{x}(\hat \lambda) \mbox{ by weak duality and feasibility of }\hat \lambda,\\
&  =   & \min \{L_x(y,z,\hat \lambda) :  (y,z) \in S  \},\\
&   =   &  \langle \hat \lambda , x-\bar x \rangle 
  + \min \{L_{\bar x}(y,z,\hat \lambda) :  (y,z) \in S  \}, \\
&   =   & \langle \hat \lambda , x-\bar x \rangle + \theta_{\bar x}(\hat \lambda),\\
&  \geq  & \langle \hat \lambda , x-\bar x \rangle + \mathcal{Q}(\bar x ) - \varepsilon_D,\\
&  \geq  &  \mathcal{C}(x) :=\langle \hat \lambda , x-\bar x \rangle + f( \hat y , \bar x)- \varepsilon_P  - \varepsilon_D.
\end{array}
$$
Moreover, since $f(\hat y, \bar x) \geq \mathcal{Q}(\bar x)$, we get
$$
\mathcal{Q}(\bar x) - \mathcal{C}(\bar x) =\varepsilon_P  + \varepsilon_D + \mathcal{Q}(\bar x) - f(\hat y , \bar x) \leq 
\varepsilon_P  + \varepsilon_D,
$$
which achieves the proof.
\end{proof}

\par As before, observe that if $\theta_{\bar x}(\hat \lambda)$ is available, i.e.,
if optimization problem \eqref{dualfunctionsecondcase} written for $x=\bar x$
and $\lambda = \hat \lambda$ is solved to optimality then
$\langle \hat \lambda , x-\bar x \rangle + \theta_{\bar x}(\hat \lambda)$
is an $\varepsilon_D$-inexact cut for $\mathcal{Q}$ at $\bar x$.

\par We also have the following corollary of 
Proposition \ref{icutsec2} that will be used
in the numerical simulations of Section \ref{sec:numexp},
offering an inexact cut easy to implement
as long as we have access to approximate primal-dual
solutions:

\begin{corollary}\label{intcor} Under the assumptions
of Proposition \ref{icutsec2}, let $\hat y$ be 
any approximate optimal and feasible solution of
primal 
 problem \eqref{pbinitprimal2}
for $x=\bar x$ and let
 $\hat \lambda$ be any approximate optimal 
feasible solution of dual problem \eqref{pbinitdual3} 
written for $x=\bar x$.
Then 
$$
\mathcal{C}(x) = \theta_{\bar x}(\hat \lambda) + 
\langle \hat \lambda , x - \bar x \rangle 
$$
is an $(f(\hat y,\bar x)-\theta_{\bar x}(\hat \lambda))$-inexact cut for $\mathcal{Q}$ at $\bar x$.
When $\hat y$ and $\hat \lambda$
are optimal solutions then we get, as expected,
an exact cut since $f(\hat y,\bar x)=\theta_{\bar x}(\hat \lambda)=\mathcal{Q}(\bar x)$.
\end{corollary}
\begin{proof} It suffices to observe that
$\hat y$ is an $\varepsilon_P$
optimal primal solution with
$\varepsilon_P=f(\hat y,\bar x)-\mathcal{Q}(\bar x)$,
that $\hat \lambda$ is an $\varepsilon_D$
optimal dual solution with
$\varepsilon_D=\mathcal{Q}(\bar x)-\theta_{\bar x}(\hat \lambda)$
and to apply Proposition \ref{icutsec2}.
\end{proof}

\par It is also worth mentioning that if we have access to an optimal primal-dual solution to \eqref{pbinitprimal2}
then we can obtain an exact cut for $\mathcal{Q}$ at $\bar x$ directly solving \eqref{pbinitprimal2} and
its dual, without adding constraint $z=x$. More precisely, a characterization of the subdifferentiable of $\mathcal{Q}$
and formulas for exact cuts for $\mathcal{Q}$ given by \eqref{pbinitprimal2} can be found in 
Lemma 2.1 in \cite{guiguessiopt2016} and Proposition 3.2 in \cite{guiguesinexactsmd}.

\section{Inexact cuts for value functions with saddle point representation of the objective} \label{sec:icut2}

The inexact cuts proposed in this section are
based on the observation that many convex functions
have saddle point representations, see for instance
\cite{nesterov03} and Section 5.6.1.1 in \cite{nemmodernconvex}.
More precisely, we assume  that the 
objective function   $f$ has a saddle point representation: if $p=(y,x)$, function $f$
is given by
\begin{equation}\label{fenchelrep}
f(p)=p^T a + \max_{w \in \mathcal{W} } \;[p^T C_0 w - \phi_0( w )]
\end{equation}
for some known convex, proper, lower semicontinuous function $\phi_0$, some known convex, compact, nonempty set $\mathcal{W}$, vector $a$,  and matrix $C_0$. In this situation, we will derive inexact cuts for $\mathcal{Q}$
without additional variables $z \in \mathbb{R}^n$ and constraints $z=x$
introduced in the previous section.

"Well structured" convex functions have 
saddle point representations, see for instance \cite{nesterov03} and Section 5.6.1.1 in \cite{nemmodernconvex} for details.
\begin{example} Function $f(p)=f(y,x)=\|y-x\|_1$ has the 
saddle point representation
$f(p)=f(y,x)=\|y-x\|_1 = \max_{\|w\|_{\infty} \leq 1} [w^T y - w^T x ]$ 
which is of form \eqref{fenchelrep} 
with $\mathcal{W}=\{w : \|w\|_{\infty} \leq 1\}$, 
$C_0=[I;-I]$, and $\phi_0$ the null function. 
\end{example}
We start considering value functions of form
\begin{equation}\label{fenchelpb1}
\mathcal{Q}(x) = \left\{ 
\begin{array}{l}
\displaystyle \min_{y \in \mathbb{R}^m} f(y,x) \\
y \in Y
\end{array}
\right.
\end{equation}
with $Y$ compact, convex, and nonempty. Let $a=[a_2;a_1]$ and let us write matrix $C_0=[A_0;B_0]$ where $A_0$ contains the first $m$ rows and
$B_0$ the last $n$ rows of $C_0$. Representation \eqref{fenchelrep} can then be written
\begin{equation}\label{fenchel2}
f(y, x) = x^T a_1 + y^T a_2 + \max_{w \in \mathcal{W}}  y^T A_0 w + x^T B_0 w - \phi_0( w )
\end{equation}
and problem \eqref{fenchelpb1} becomes the saddle point problem 
\begin{equation}\label{spp}
\mathcal{Q}( x )  = 
\displaystyle \min_{y \in Y} \max_{w \in \mathcal{W}} \;x^T a_1 + y^T a_2 +
y^T A_0 w + x^T B_0 w - \phi_0(w).
\end{equation}
Since $Y$ and $\mathcal{W}$ are convex, compact and nonempty, this saddle point problem can be equivalently written as 
the convex problem
\begin{equation}\label{conv1}
\mathcal{Q}(x)=x^T a_1 + \left\{ 
\begin{array}{l}
\displaystyle \max_{w} \; \theta_x(w ) \\ 
w \in \mathcal{W}
\end{array}
\right.
\end{equation}
where concave function $\theta_x$ is given by
\begin{equation}\label{conv2}
\theta_x(w ) = \left\{ 
\begin{array}{l}
\displaystyle \min_{y} \;L_x(y, w) \\
y \in Y,
\end{array}
\right.
\end{equation}
where 
\begin{equation}\label{lagrangianfirst}
L_x(y, w) =y^T (a_2 + A_0 w )  + x^T B_0 w - \phi_0( w ).
\end{equation}

Once again, the linearity in $x$ of this new Lagrangian function $L_x(y,w)$ will
allow us to derive inexact cuts. However, contrary to the previous section,
this linearity was achieved using a saddle point representation of $f$.
The following proposition provides inexact cuts for $\mathcal{Q}$
given by \eqref{fenchelpb1} with $f$ of the form \eqref{fenchel2}.
\begin{proposition}\label{inexutsFenchelobj} Consider problem 
\eqref{fenchelpb1} with $f$ having a saddle point representation of form \eqref{fenchel2}. 
Assume that $Y$ and $\mathcal{W}$ are compact, convex, and nonempty.
Let $\hat w \in \mathcal{W}$
be an $\varepsilon$-optimal solution of problem \eqref{conv1} written with $x=\bar x$ and
let $\hat y \in Y$ be a $\tau$-optimal solution of problem \eqref{conv2} written with 
$x=\bar x, w = \hat w$. 
Then the affine function 
\begin{equation}\label{cutndiff1}
\mathcal{C}(x):=x^{\top} \Big( a_1 + B_0 \hat w \Big) + {\hat y}^{\top}  \Big(a_2 + A_0 \hat w \Big) -\phi_0(  \hat w   ) - \tau
\end{equation}
is a $(\varepsilon + \tau)$-inexact cut for $\mathcal{Q}$ at $\bar x$.
\end{proposition}
\begin{proof} Let $(\bar y , \bar w)$ be an optimal solution of saddle point problem \eqref{spp} with $x=\bar x$.
By definition of $\hat w$ and $\hat y$, we have 
\begin{equation}\label{firstrel}
\theta_{\bar x}(\bar w ) - \varepsilon \leq \theta_{\bar x}( \hat w ) \mbox{ and }
\theta_{\bar x}( \hat w ) + \tau \geq L_{\bar x}( \hat y , \hat w ) \geq \theta_{\bar x}(\hat w ).
\end{equation}
By linearity of $L_{\cdot}(y,w)$ we get for every $y \in Y, w \in \mathcal{W}$, that
\begin{equation}\label{linL}
L_x( y , w ) =
L_{\bar x}(y , w) + ( x - \bar x )^T B_0 w.
\end{equation}
Next, using representation \eqref{conv1} of $\mathcal{Q}$ and the fact that $\hat w \in \mathcal{W}$ we have 
$$
\begin{array}{lcl}
\mathcal{Q}(x) & \geq &  x^T a_1  + \theta_x( \hat w )\\
& = &  x^T a_1  + \left\{ 
\begin{array}{l}
\min \;L_x(y, \hat w) \\
y \in Y,
\end{array}
\right.\\
& \stackrel{\eqref{linL}}{=} &  x^T a_1  + ( x - \bar x )^T B_0 \hat w   + \left\{ 
\begin{array}{l}
\min \;L_{\bar x}(y, \hat w) \\
y \in Y,
\end{array}
\right.\\
& = &  x^T a_1  + ( x - \bar x )^T B_0 \hat w +  \theta_{\bar x}( \hat w )\\
& \stackrel{\eqref{firstrel}}{\geq } &  x^T a_1  +  ( x - \bar x )^T B_0 \hat w + L_{\bar x}( \hat y , \hat w ) - \tau \\
& \stackrel{\eqref{cutndiff1}}{=} &  \mathcal{C}(x).
\end{array}
$$
Moreover,
$$
0 \leq \mathcal{Q}( \bar x ) - \mathcal{C}( \bar x)= \tau + \theta_{\bar x}(\bar w ) - L_{\bar x}( \hat y , \hat w ) 
 \stackrel{\eqref{conv2}}{\leq} \tau +  \theta_{\bar x}(\bar w ) - \theta_{\bar x}( \hat  w ) \leq  \tau + \varepsilon,
$$
which achieves the proof of the proposition.
\end{proof}
Now consider value function $\mathcal{Q}$ given by
\begin{equation}\label{fvaleurlinconstr}
\mathcal{Q}( x )  = 
\left\{
\begin{array}{l}
\min f(y,x) \\
y \in Y,\; Ay + B x = b,
\end{array}
\right.
\end{equation}
with $Y$ convex, nonempty, and compact.
If $f$ has a saddle point representation of form \eqref{fenchel2}
with $\mathcal{W}$ convex, nonempty, and compact, value function \eqref{fvaleurlinconstr} can be written 
\begin{equation}\label{fvaleurlinconstrfenchel}
\mathcal{Q}( x )  = 
x^T a_1 + 
\left\{
\begin{array}{l}
\max \theta_x(w) \\
w \in \mathcal{W}
\end{array}
\right.
\end{equation}
where 
\begin{equation}\label{thetaxwnew}
\theta_x( w ) = 
\left\{
\begin{array}{l}
\min \;y^T( a_2 + A_0 w ) + x^T B_0 w - \phi_0( w ) \\
y \in Y, Ay + B x = b.
\end{array}
\right. 
\end{equation}
For problem \eqref{thetaxwnew} define the Lagrangian
\begin{equation}
\begin{array}{lcl}
\mathcal{L}_{x, w }(y, \lambda) &  = &y^T( a_2 + A_0 w ) + x^T B_0 w - \phi_0( w ) +  \lambda^T ( Ay + Bx - b)\\
&=& L_x(y,w) + \lambda^T ( Ay + Bx - b), 
\end{array}
\end{equation}
where $L_x(y,w)$ is given by \eqref{lagrangianfirst}. 
Let us fix $\bar x \in \mathbb{R}^n$ and 
assume that there is $y_0 \in \mbox{ri}(Y)$ such that
$A y_0 + B \bar x = b$. Then by the Convex Duality theorem,  we can express $\theta_{\bar x}( w )$
as the optimal value of the dual of \eqref{thetaxwnew}:
\begin{equation}\label{thetaxreph}
\theta_{\bar x}( w ) = 
\displaystyle \max_{\lambda } \;h_{{\bar x } , w}(\lambda )
\end{equation}
for the dual function
\begin{equation}\label{defhxw}
h_{\bar x, w}(\lambda ) = \left\{ 
\begin{array}{l}
\min \;\mathcal{L}_{\bar x, w}(y, \lambda)\\
y \in Y.
\end{array}
\right.
\end{equation}

\begin{proposition}\label{inexutsFenchelobj2} Consider problem \eqref{fvaleurlinconstr} with $f$ having a saddle point representation of form \eqref{fenchel2}. Assume that sets $Y$ and $\mathcal{W}$ are nonempty, convex, and compact.
Let us fix $\bar x \in \mathbb{R}^n$ and 
assume that there is $y_0 \in \mbox{ri}(Y)$ such that
$A y_0 + B \bar x = b$.
Let $(\bar y , \bar w)$ be an optimal solution of saddle point problem \eqref{fvaleurlinconstrfenchel}  with $x=\bar x$
and let $\hat w \in \mathcal{W}$ 
be an $\varepsilon$-optimal solution of problem \eqref{fvaleurlinconstrfenchel} written with $x=\bar x$:
\begin{equation}\label{firstapprox}
\theta_{\bar x}( \hat w ) \geq \theta_{\bar x}( \bar w ) - \varepsilon, 
\end{equation}
and let $\hat \lambda \in Y$ be a $\delta$-optimal solution of problem  
$$
\theta_{\bar x}( \hat w ) = 
\displaystyle \max_{\lambda} \;h_{\bar x, \hat w}(\lambda ) 
$$
i.e.,
\begin{equation}\label{secondapprox}
h_{\bar x, \hat w}(  \hat \lambda  ) \geq \theta_{\bar x}( \hat w) - \delta.  
\end{equation}
Let $\hat y$ be a $\tau$-optimal feasible solution of 
$$
\theta_{\bar x}( \hat w ) = 
\left\{
\begin{array}{l}
\min \;y^T( a_2 + A_0 \hat w ) + {\bar x}^T B_0 \hat w - \phi_0( \hat w ) \\
y \in Y, Ay + B \bar x = b,
\end{array}
\right. 
$$
i.e.,
\begin{equation}\label{caractyeps}
\hat y \in Y,\; A \hat y + B \bar x = b,\;L_{\bar x}( \hat y, \hat w ) \leq \theta_{\bar x}( \hat w ) + \tau. 
\end{equation}
Then the affine function 
\begin{equation}\label{cutndiff2}
\mathcal{C}(x)=x^T \Big( a_1 + B_0 \hat w + B^T {\hat \lambda} \Big) + {\hat y}^T  \Big(a_2 + A_0 \hat w \Big)
- {\bar x}^T B^T \hat \lambda -\phi_0(  \hat w  ) -\tau -\delta
\end{equation}
is a $(\varepsilon + \tau + \delta)$-inexact cut for 
$\mathcal{Q}$ at $\bar x$.
\end{proposition}
\begin{proof} By linearity of $\mathcal{L}_{\cdot,w}(y, \lambda)$ we get for every $y \in Y, w \in \mathcal{W}$, that
\begin{equation}\label{linL2}
\mathcal{L}_{x,w}( y , \lambda ) =
\mathcal{L}_{\bar x,w}( y , \lambda ) + ( x - \bar x )^T ( B_0 w + B^T \lambda ).
\end{equation}
Next, using representation \eqref{fvaleurlinconstrfenchel} of $\mathcal{Q}$ and the fact that $\hat w \in \mathcal{W}$ we have 
$$
\begin{array}{lcl}
\mathcal{Q}(x) & \geq &  x^T a_1  + \theta_{x}( \hat w )\\
 & \geq &  x^T a_1  + h_{x, \hat w}( \hat \lambda ),\\
& \stackrel{\eqref{defhxw}}{=} &  x^T a_1  + 
 \left\{ 
\begin{array}{l}
\min \;\mathcal{L}_{x, \hat w}(y, \hat \lambda) \\
y \in Y,
\end{array}
\right.\\
& \stackrel{\eqref{linL2}}{= } &  x^T a_1  +  ( x - \bar x )^T ( B_0 \hat w + B^T \hat \lambda   )  +
\left\{ 
\begin{array}{l}
\min \;\mathcal{L}_{\bar x, \hat w}(y, \hat \lambda) \\
y \in Y,
\end{array}
\right.\\
& \stackrel{\eqref{defhxw}}{=} &  x^T a_1  +  ( x - \bar x )^T ( B_0 \hat w + B^T \hat \lambda   ) + h_{\bar x , \hat w}( \hat \lambda )\\
& \stackrel{\eqref{secondapprox}}{\geq} & x^T a_1  +  ( x - \bar x )^T ( B_0 \hat w + B^T \hat \lambda   ) + \theta_{\bar x}( \hat w )- \delta \\
& \stackrel{\eqref{caractyeps}}{\geq} & x^T a_1  +  ( x - \bar x )^T ( B_0 \hat w + B^T \hat \lambda   ) + L_{\bar x}(\hat y , \hat w) -\tau - \delta \\
& \stackrel{\eqref{cutndiff2}}{=} &  \mathcal{C}(x).
\end{array}
$$
Moreover, if $\bar w$ is an optimal solution of \eqref{fvaleurlinconstrfenchel} written for $x=\bar x$, i.e., 
$\mathcal{Q}( \bar x) = \bar x^T a_1 + \theta_{\bar x}( \bar w)$ we obtain
$$
0 \leq \mathcal{Q}( \bar x ) - \mathcal{C}( \bar x)= \tau + \delta + \theta_{\bar x}(\bar w ) - L_{\bar x}( \hat y , \hat w ) 
 \stackrel{\eqref{thetaxwnew}}{\leq} \tau + \delta +  \theta_{\bar x}(\bar w ) - \theta_{\bar x}( \hat  w ) \stackrel{\eqref{firstapprox}}{\leq}  \tau + \delta + \varepsilon,
$$
which achieves the proof of the proposition.
\end{proof}

\section{Particular case of differentiable problems and comparison with the inexact cuts from \cite{guigues2016isddp}}
\label{sec:comparison}

The following proposition, taken from \cite{guigues2016isddp}, provides an inexact cut for $\mathcal{Q}$ given by \eqref{pbinitprimal2}
when functions $f, g_i$ are differentiable.
\begin{proposition} \label{varprop1}
Consider value function $\mathcal{Q}$ given by \eqref{pbinitprimal2}.
Let Assumption (H0) hold, take $\bar x \in X$, and assume that
\begin{equation}\label{slaterdiff}
\mbox{there exists }y_{\bar x} \in \mbox{ri}(Y)
\mbox{ such that }A y_{\bar x} + B \bar x =b \mbox{ with }g(y_{\bar x},\bar x)<0.
\end{equation}
Assume that $f$ and $g$ are differentiable on $Y \small{\times} X$. 
Let  $\varepsilon \geq 0$,
let $\hat y$ be an $\epsilon$-optimal feasible primal solution for problem \eqref{pbinitprimal2}
written for $x= \bar x$ 
and let $(\hat \lambda, \hat \mu)$ be an $\epsilon$-optimal feasible solution of the
corresponding dual problem given by 
$$
\max_{\mu \geq 0, \lambda} \theta_{x}(\lambda , \mu)
$$
where the dual function $\theta_{x}(\lambda , \mu)$ is given by
\begin{equation}\label{defthetaxdiff}
\theta_{x}(\lambda , \mu) = \min_{y \in Y} L_x(y,\lambda,\mu)
\end{equation}
for the Lagrangian
$$
L_x(y,\lambda,\mu)=f(y,x) + \langle \lambda , Bx + Ay -b\rangle + 
\langle \mu , g(y,x) \rangle.
$$
Assume that $f(\cdot,\bar x)$ is finite on 
\begin{equation}\label{Sxdef}
S(\bar x)=\{y \in Y : Ay + B\bar x=b, g(y,\bar x)\leq 0\}
\end{equation}
and that 
$\eta(\varepsilon)=\ell (\hat y , \bar x  , \hat \lambda , \hat \mu )$ is finite
where 
$$
\ell(\hat y , \bar x  , \hat \lambda , \hat \mu )=\max \{ \langle \nabla_y L_{\bar x}(\hat y , \hat \lambda , \hat \mu)  ,   \hat y   - y \rangle   : y \in Y \}.
$$
Then the affine function
\begin{equation}\label{cutvarprop1}
\mathcal{C}(x):= L_{\bar x} ( \hat y, {\hat \lambda}, \hat \mu )- \eta(\varepsilon)  + 
\langle \nabla_x L_{\bar x} ( \hat y, {\hat \lambda}, \hat \mu ) , x - \bar x \rangle
\end{equation}
is an $(\varepsilon + \ell (\hat y,  \bar x  , \hat \lambda , \hat \mu ))$-inexact
cut for $\mathcal{Q}$ at $\bar x$.
\end{proposition}

\par We want to compare the inexact cuts given by Propositions \ref{icutsec2} and \ref{varprop1} obtained
taking $\varepsilon_D=\varepsilon_P=\varepsilon$ in Proposition \ref{icutsec2}.
For the cut given by Proposition \ref{varprop1} to be valid, we assume that the assumptions
of this proposition are satisfied. In particular, \eqref{slaterdiff} holds. Let us show that
if in addition $Y \times X \subset \mbox{dom}(g_i)$ for all $i=1,\ldots,p$,
this implies that \eqref{slater} holds 
which will imply that the assumptions of Proposition \ref{icutsec2}
are also satisfied and the inexact cut given by that proposition is valid.
Indeed, write set $S$ given by \eqref{defS} as
$S=S_1 \cap S_2 \cap (Y \times \mathbb{R}^n)$
where $S_1=\{(y,z) \in \mathbb{R}^m \small{\times} \mathbb{R}^n : g(y,z) \leq 0\}$ and
$S_2 = \{(y,z) \in \mathbb{R}^m \small{\times} \mathbb{R}^n : A y + B z = b\}$.
We have that $\mbox{ri}(S_2)=S_2$ and
\begin{equation}\label{rigi}
\mbox{ri}(\{g_i \leq 0\})=\{(y,z) \in \mathbb{R}^m \small{\times} \mathbb{R}^n : 
(y,z) \in \mbox{ri}(\mbox{dom}(g_i)), g_i(y,z)<0,i=1,\ldots,p\}.
\end{equation}
Since $Y \times \{\bar x\} \subset \mbox{dom}(g_i), i=1,\ldots,p$,
we have $\mbox{ri}(Y) \times \{\bar x\} \subset \mbox{ri}(\mbox{dom}(g_i)), i=1,\ldots,p$,
implying that set $\cap_{i=1}^p \mbox{ri}(\{g_i \leq 0\})$ is nonempty since it
contains the nonempty set $\mbox{ri}(Y) \times \{\bar x\}$ (this set contains
$(y_{\bar x}, \bar x)$).
Therefore  $\mbox{ri}(S_1) = \cap_{i=1}^p \mbox{ri}(\{g_i \leq 0\})
=\{(y,z) \in \mathbb{R}^m \small{\times} \mathbb{R}^n : 
(y,z) \in \mbox{ri}(\mbox{dom}(g_i)), g_i(y,z)<0, i=1,\ldots,p\}$. 
It follows that convex sets $S_1, S_2$, and $Y \times \mathbb{R}^n$ are convex and satisfy
$\mbox{ri}(S_1) \cap \mbox{ri}(S_2) \cap (\mbox{ri}(Y)\times \mathbb{R}^n) \neq \emptyset$ (they contain
the point $(y_{\bar x}, \bar x)$) which implies that 
$\mbox{ri}(S)=\mbox{ri}(S_1) \cap \mbox{ri}(S_2) \cap (\mbox{ri}(Y)\times \mathbb{R}^n)$
and recalling the representations of $\mbox{ri}(S_1)$ and $\mbox{ri}(S_2)$, we see 
that $(y_{\bar x},\bar x)$ which satisfies \eqref{slaterdiff} also belongs to $\mbox{ri}(S)$, i.e.,
Slater condition \eqref{slater} holds. Therefore, Proposition \ref{icutsec2} provides
a valid $2\varepsilon$-inexact cut for $\mathcal{Q}$.

Let us use the notation $\mathcal{C}_1(x)=\theta_1 + \langle \beta_1 , x - \bar x \rangle $
and $\mathcal{C}_2(x)={\theta}_2 + \langle {\beta}_2 , x - \bar x \rangle $
for respectively 
the inexact cuts given by Propositions \ref{icutsec2} and \ref{varprop1}.
In Proposition \ref{propcompcuts} below, we derive upper and lower bounds on 
$\theta_1-\theta_2=\mathcal{C}_1(\bar x)-\mathcal{C}_2(\bar x)$ (observe that in the exact case, i.e., when $\varepsilon=0$, clearly
$\theta_1=\theta_2$ and $\beta_1=\beta_2$).
This will be done using characterizations of $\varepsilon$-optimal feasible primal-dual solutions
to obtain bounds for the terms $\langle \hat \mu , g(\hat y , \bar x ) \rangle$ 
and $\displaystyle \max_{y \in Y} 
\langle \nabla_y L_{\bar x}(\hat y,\hat \lambda,\hat \mu),  \hat y - y \rangle$ (which are clearly null if $\hat y$ and $(\hat \lambda, \hat \mu)$ are optimal
primal-dual solutions). In particular, we will show that $\langle \hat \mu , g(\hat y , \bar x ) \rangle$
is between $-2\varepsilon$ and 0.  
To derive these bounds, we will assume that\\ 
\begin{itemize}
\item[(A0)] the gradient of objective function $f(\cdot,\bar x)$ (resp. 
of constraint function $g_i(\cdot,\bar x)$) is $L_0$ (resp. $L_i$)-co-coercive with $L_i>, i=0,\ldots,p$. \\
\end{itemize}
Recall that $F:\mbox{Dom}(F) \subseteq \mathbb{R}^m \rightarrow \mathbb{R}^m$ is $L$-co-coercive
on $\Omega \subseteq \mbox{Dom}(F)$ if 
$$
L \langle y - x , F(y) - F(x) \rangle \geq \|F(y)-F(x)\|^2,\;\forall x,y \in \Omega.
$$
\begin{proposition}\label{propcompcuts}
Let the assumptions of Proposition \ref{varprop1} hold and assume that 
$Y \times X \subset \mbox{dom}(g_i)$ for all $i=1,\ldots,p$.
Take $\bar x \in X$ and let $\mathcal{L}_{\bar x}$ be any lower bound 
on $\mathcal{Q}(\bar x)$. Let
$\mathcal{C}_1(x)=\theta_1 + \langle \beta_1 , x - \bar x \rangle $
and $\mathcal{C}_2(x)={\theta}_2 + \langle {\beta}_2 , x - \bar x \rangle $
be respectively the inexact cuts given by Propositions \ref{icutsec2} and \ref{varprop1}
taking $\varepsilon_D=\varepsilon_P=\varepsilon$.
Assume that $f$ and $g_i,i=1,\ldots,p$, satisfy (A0), that $Y$ is compact, and set  
$$
\mathcal{U}_{\bar x}= 
\frac{f(y_{\bar x} , \bar x  )  - \mathcal{L}_{\bar x} + \varepsilon }{\min (-g_{i}(y_{\bar x} ,  \bar x), i=1,\ldots,p )},
L = L_0 +  \mathcal{U}_{\bar x} \max_{i=1,\ldots,p} L_i.
$$
Then we have
$$
-2\varepsilon \leq \mathcal{C}_1( \bar x ) -  \mathcal{C}_2( \bar x ) 
\leq 2\varepsilon + 2 D_Y \sqrt{L \varepsilon},
$$
where $D_Y$ is the diameter of $Y$.
\end{proposition}
\begin{proof} Recall that 
$$
\begin{array}{l}
\mathcal{C}_1( \bar x ) = f(\hat y , \bar x) - 2\varepsilon,\\
\mathcal{C}_2( \bar x ) = f(\hat y , \bar x) + \langle \hat \mu , g(\hat y , \bar x ) \rangle 
- \displaystyle \max_{y \in Y} 
\langle \nabla_y L_{\bar x}(\hat y,\hat \lambda,\hat \mu),  \hat y - y \rangle,
\end{array}
$$
and that $(\hat y, \hat \lambda, \hat \mu)$ satisfy 
\begin{equation} \label{eq:approx-sol}
\hat y \in S(\bar x),\; \hat \mu \geq 0,\; f(\hat y,\bar x) \leq \mathcal{Q}(\bar x) + \varepsilon,\;
\theta_{\bar x}(\hat \lambda , \hat \mu) \geq \mathcal{Q}(\bar x)-\varepsilon,
\end{equation}
where  $S(x)$ is defined in \eqref{Sxdef} and $\theta_{\bar x}$ is the dual function given by 
\eqref{defthetaxdiff}.

By the subgradient inequality, if $L_x$ is the Lagrangian given in 
Proposition \ref{varprop1}, we get
\begin{equation}\label{ineqthetadiff1}
\begin{array}{lcl}
\theta_{\bar x}(\hat \lambda,\hat \mu) &= & \min_{y \in Y} L_{\bar x}(y,\hat \lambda,\hat \mu)
\ge L_{\bar x}(\hat y,\hat \lambda,\hat \mu)
+ \min_{y \in Y} \inner{\nabla_y L_{\bar x}(\hat y,\hat \lambda,\hat \mu)}{y - \hat y} \\
 &=& f(\hat y, \bar x)  + \inner{\hat \mu}{g(\hat y,\bar x)}
 + \min_{y \in Y} \inner{\nabla_y L_{\bar x}(\hat y,\hat \lambda,\hat \mu)}{y - \hat y} = \mathcal{C}_2( \bar x ).
 \end{array}
 \end{equation}
Therefore,
\begin{equation}\label{firstineqcomp}
  \begin{array}{lcl}
 \mathcal{C}_1( \bar x ) & = & f(\hat y , \bar x) - 2\varepsilon \\
 & \geq   & \theta_{\bar x}(\hat \lambda,\hat \mu) - 2\varepsilon \mbox{ by weak duality},\\
 & \stackrel{\eqref{ineqthetadiff1}}{\geq} & \mathcal{C}_2( \bar x ) -2\varepsilon. 
\end{array}
\end{equation} 
We next provide an upper bound for 
$\mathcal{C}_1(\bar x)-\mathcal{C}_2(\bar x)$. Indeed, \eqref{eq:approx-sol} implies that
\[
f(\hat y, \bar x)   \le \theta_{\bar x}(\hat \lambda, \hat \mu) + 2 \varepsilon = \min _{y \in Y}  \{ L_{\bar x}(y,\hat \lambda,\hat \mu)  : y \in Y \} + 2\varepsilon
\]
and hence that
\[
L_{\bar x}(\hat y,\hat \lambda,\hat \mu) = f(\hat y, \bar x) + \inner{\hat \mu}{ g(\hat y , \bar x)} \le  
\min _{y \in Y}  \{ L_{\bar x}(y,\hat \lambda,\hat \mu)  : y \in Y \} + \inner{\hat \mu}{ g(\hat y , \bar x)} +  2\varepsilon
\]
where the first equality is due to $\hat y \in S(\bar x)$.
The last inequality in turn is equivalent to  $\tilde \varepsilon := 2 \varepsilon +  \inner{\hat \mu}{ g(\hat y, \bar x)}$ satisfying
\begin{equation}\label{ineqepstilde}
2\varepsilon \geq
2 \varepsilon +  \inner{\hat \mu}{ g(\hat y, \bar x)}=
\tilde \varepsilon \ge 0, \quad 0 \in \partial_{\tilde \varepsilon} \left( L_{\bar x}(\cdot,\hat \lambda,\hat \mu)  + \delta_Y(\cdot) \right)(\hat y)
\end{equation}
%
%
where $\delta_Y(\cdot)$ is the indicator function of set $Y$ given by
$$
\delta_{Y}(y) = \left\{
\begin{array}{ll}
0  & \mbox{if }y \in Y,\\
+\infty & \mbox{otherwise.}
\end{array}
\right.
$$
It is easy to check that $\|\hat \mu\| \leq \mathcal{U}_{\bar x}$
(see for instance the proof of Proposition 2.3 in \cite{guigues2016isddp})
which easily implies that $L_{\bar x}(\cdot,\hat \lambda,\hat \mu)$ 
is $L$-co-coercive (for the interested reader, we provide in Lemma \ref{lemmasumcocoercive} in the appendix
the proof that
a sum of $L_i$-co-coercive  mappings $f_i$ is ($\sum_{i=1}^n L_i$)-co-coercive).
Combining this observation with \eqref{ineqepstilde}
and Lemma 3.2 in \cite{yunlongmonteiro}, we obtain that 
there exists $v$ satisfying:
\begin{equation}\label{satsfv}
v \in \nabla_y L_{\bar x}(\hat y,\hat \lambda,\hat \mu) + \partial_{\tilde \varepsilon} \delta_Y( \hat y ), \quad \|v\| \le \sqrt{2 L \tilde \varepsilon} \stackrel{\eqref{ineqepstilde}}{\leq}  2 \sqrt{L \varepsilon}.
\end{equation}
It is well known that set $\partial_{\tilde \varepsilon} \delta_Y( \hat y )$ is the 
$\tilde \varepsilon$-normal set to $Y$ at $\hat y$
given by 
$$
\partial_{\tilde \varepsilon} \delta_Y( \hat y ) = \{
z \in \mathbb{R}^m : \langle z , y - \hat y \rangle \leq \tilde \varepsilon \;\forall y \in Y\}
$$
and therefore $v$ which satisfies \eqref{satsfv} also satisfies
\begin{equation}\label{maxscalpos}
\langle \nabla_y L_{\bar x}(\hat y,\hat \lambda,\hat \mu) - v ,  \hat y - y \rangle \leq \tilde \varepsilon,\;
\forall y \in Y
\Leftrightarrow \max_{y \in Y} \langle  \nabla_y L_{\bar x}(\hat y,\hat \lambda,\hat \mu) - v ,  \hat y - y \rangle \leq \tilde \varepsilon.
\end{equation}
We then obtain the following upper bound for $\mathcal{C}_1(\bar x)-\mathcal{C}_2(\bar x)$:
$$
\begin{array}{lcl}
 \mathcal{C}_2(\bar x) &= &f(\hat y , \bar x) +
 \langle \hat \mu , g(\hat y , \bar x ) \rangle
 - \displaystyle \max_{y \in Y} 
\langle \nabla_y L_{\bar x}(\hat y,\hat \lambda,\hat \mu),  \hat y - y \rangle \\
&= &\mathcal{C}_1(\bar x) + 2 \varepsilon + \langle \hat \mu , g(\hat y , \bar x ) \rangle
 - \displaystyle \max_{y \in Y} 
\langle \nabla_y L_{\bar x}(\hat y,\hat \lambda,\hat \mu),  \hat y - y \rangle \\
& \stackrel{\eqref{ineqepstilde}}{\geq} &\mathcal{C}_1(\bar x) 
 - \displaystyle \max_{y \in Y} 
\langle \nabla_y L_{\bar x}(\hat y,\hat \lambda,\hat \mu),  \hat y - y \rangle \\
& \ge & \mathcal{C}_1(\bar x) 
 - \displaystyle \max_{y \in Y} 
\langle \nabla_y L_{\bar x}(\hat y,\hat \lambda,\hat \mu) - v ,  \hat y - y \rangle
-  \displaystyle \max_{y \in Y} 
\langle v ,  \hat y - y \rangle\\
& \stackrel{\eqref{maxscalpos}}{\geq } & \mathcal{C}_1(\bar x)  - \tilde \varepsilon
 - \|v\| \, D_Y \stackrel{\eqref{satsfv}}{\geq}
 \mathcal{C}_1(\bar x)  - 2 \varepsilon
 - 2 D_Y  \sqrt{L \varepsilon},
\end{array} 
 $$
 which achieves the proof of the proposition.
 \end{proof} 
 
The upper and lower bounds on $\mathcal{C}_1(\bar x) - \mathcal{C}_2(\bar x)$ given in Proposition \ref{propcompcuts} are continuous functions
of $\varepsilon$ which go to $0$ as $\varepsilon$ goes to $0$. Also these bounds are respectively
positive and negative for positive $\varepsilon$. This shows that they are both of good quality for small
values of $\varepsilon$ and this analysis does not ensure that one of these two is always better (i.e., has a
larger intercept at $\bar x$) than the other. 

The analysis above (the proof of Proposition \ref{propcompcuts}) is also interesting per-se since it offers ways of
characterizing $\varepsilon$-optimal primal-dual solutions and allows us to derive bounds on the two 
quantities $\langle \hat \mu , g(\hat y , \bar x ) \rangle$ and
$\displaystyle \max_{y \in Y} 
\langle \nabla_y L_{\bar x}(\hat y,\hat \lambda,\hat \mu),  \hat y - y \rangle$
which, by the first order optimality conditions, are null if $\hat y$ and $(\hat \lambda, \hat \mu)$
are respectively optimal primal and dual solutions.
More precisely, if  $\hat y$ (resp. $(\hat \lambda, \hat \mu)$)
is an $\varepsilon$-optimal  feasible primal (resp. dual) solution, then we
have shown that $-2\varepsilon \leq \langle \hat \mu , g(\hat y , \bar x ) \rangle \leq 0$
and $0 \leq \displaystyle \max_{y \in Y} 
\langle \nabla_y L_{\bar x}(\hat y,\hat \lambda,\hat \mu),  \hat y - y \rangle \leq 2D_Y \sqrt{L \varepsilon} + 2\varepsilon$.

\section{ISDDP algorithm for nondifferentiable problems}\label{sec:isddp}

The objective of this section is to introduce and study new variants of ISDDP
which use the inexact cuts built in the previous sections.

We consider multistage stochastic nonlinear optimization problems of the form
\begin{equation}\label{pbtosolve}
\begin{array}{l}
\displaystyle 
\min_{x_1 \in X_1( x_0 , \xi_1)} f_1(x_1,x_0,\xi_1 ) +
\mathbb{E}\left[ \min_{x_2 \in X_2( x_1 , \xi_2)} f_2(x_2,x_1,\xi_2) +
\mathbb{E}\left[ \ldots \right. \right.  \\
\hspace*{4.6cm} \left. \left. \ldots + \mathbb{E}\left[ \min_{x_T \in X_T( x_{T-1} , \xi_T)} f_T(x_{T},x_{T-1},\xi_T) \right] \right] \right],
\end{array}
\end{equation}
where $x_0$ is given,  $(\xi_t)_{t=2}^T$ is a stochastic process, $\xi_1$ is deterministic, and
$$ 
X_t( x_{t-1} , \xi_t)= \{x_t \in \mathbb{R}^n : \displaystyle A_{t} x_{t} + B_{t} x_{t-1} = b_t,
g_t(x_t, x_{t-1}, \xi_t) \leq 0, x_t \in \mathcal{X}_t \}.
$$

We make the following assumption on $(\xi_t)$:\\
\par (H) $(\xi_t)$ 
is interstage independent and
for $t=2,\ldots,T$, $\xi_t$ is a random vector taking values in $\mathbb{R}^K$ with a discrete distribution and
a finite support $\Theta_t=\{\xi_{t 1}, \ldots, \xi_{t N_t}\}$ with $p_{t i}=\mathbb{P}(\xi_t = \xi_{t i})>0,i=1,\ldots,N_t$, while $\xi_1$ is deterministic.\\

\par In the sequel, we will denote by $A_{t j}$, $B_{t j}$, and $b_{t j}$ the realizations of $A_t, B_t$,
and $b_t$ in $\xi_{t j}$.

For this problem, we can write Dynamic Programming equations: the first stage problem is 
\begin{equation}\label{firststodp}
\mathcal{Q}_1( x_0 ) = \left\{
\begin{array}{l}
\min_{x_1 \in \mathbb{R}^n} f_1(x_1, x_0, \xi_1)  + \mathcal{Q}_2 ( x_1 )\\
x_1 \in X_1( x_{0}, \xi_1 )\\
\end{array}
\right.
\end{equation}
for $x_0$ given and for $t=2,\ldots,T$, $\mathcal{Q}_t( x_{t-1} )= \mathbb{E}_{\xi_t}[ \mathfrak{Q}_t ( x_{t-1},  \xi_{t}  )  ]$ with
\begin{equation}\label{secondstodp} 
\mathfrak{Q}_t ( x_{t-1}, \xi_{t}  ) = 
\left\{ 
\begin{array}{l}
\min_{x_t \in \mathbb{R}^n}  f_t ( x_t , x_{t-1}, \xi_t ) + \mathcal{Q}_{t+1} ( x_t )\\
x_t \in X_t ( x_{t-1}, \xi_t ),
\end{array}
\right.
\end{equation}
with the convention that $\mathcal{Q}_{T+1}$ is null.

We set $\mathcal{X}_0=\{x_0\}$ and make the following assumptions (H1) on the problem data:\\
\par (H1): there exists $\varepsilon>0$ such that for $t=1,\ldots,T$,  
\begin{itemize}
\item[1)] $\mathcal{X}_{t}$ is a nonempty, compact, and convex set.
\item[2)] For every $j=1,\ldots,N_t$, the function
$f_t(\cdot, \cdot , \xi_{t j})$ is convex, proper, lower semicontinuous on $\mathcal{X}_t \small{\times} \mathcal{X}_{t-1}$ 
and for every $x_{t-1}$
$\mathcal{X}_{t-1}^{\varepsilon}$ we have 
$$
\mathcal{X}_t \subset \dom ( f_t(\cdot, x_{t-1} , \xi_{t j})).
$$
\item[3)] For every $j=1,\ldots,N_t$, each component $g_{t i}(\cdot, \cdot, \xi_{t  j}), i=1,\ldots,p$, of function $g_{t}(\cdot, \cdot, \xi_{t  j})$ is
convex, lower semicontinuous and finite on $\mathcal{X}_t \small{\times }{\mathcal{X}_{t-1}}$.
\item[4)] $X_1(x_0,\xi_1) \neq \emptyset$ and for every $t=2,\ldots,T$, 
for every $j=1,\ldots,N_t$, for every $x_{t-1} \in \mathcal{X}_{t-1}^{\varepsilon}$, the set 
$\mbox{ri}( \mathcal{X}_t) \cap X_t(x_{t-1},\xi_{t j})$ is nonempty.
\item[5)] for every $t \geq 2$, for every $j=1,\ldots,N_t$, there is 
$(x_{t j},x_{t-1 j})\in \mbox{ri}(\mathcal{X}_t) \small{\times} \mathcal{X}_{t-1}$
such that $g_t(x_{t j},x_{t-1 j},\xi_{t j})<0$.\\
\end{itemize}

We are now in a position to describe the ISDDP algorithm for nondifferentiable optimization problems
of form \eqref{pbtosolve}. The ISDDP algorithm given below combines SDDP with the inexact cuts
derived in Section \ref{sec:gencut}:\\
\rule{\linewidth}{1pt}
\par {\textbf{ISDDP algorithm.}}\\
\rule{\linewidth}{1pt}
\begin{itemize}
\item[Step 0)] {\textbf{Initialization.}} Let $\mathcal{Q}^{0}_t : \mathcal{X}_{t-1} \to \mathbb{R},\,t=2,\ldots,T+1$, 
be affine functions satisfying $\mathcal{Q}_t^0 \leq \mathcal{Q}_t$. Set $k=1$.
\item[Step 1)] {\textbf{Forward pass.}} Setting $x_0^{k} = x_0$, generate
a sample $(\tilde \xi_1^k,\tilde \xi_2^k,\ldots,\tilde \xi_T^k)$ from the distribution of $(\xi_1,\xi_2,\ldots,\xi_T)$ and for $t=1,2,\ldots,T$,
compute a $\delta_t^k$-optimal solution $x_t^k$  of
\begin{align}
\min \left\{ f_t(x_t, x_{t-1}^{k} , {\tilde \xi}_t^k ) + \mathcal{Q}^{k-1}_{t+1}( x_t ) : x_t \in X_t( x_{t-1}^{k} , \tilde \xi_t^k )    \right \}.
\end{align}
\item[Step 2)] {\textbf{Backward pass.}}
\par For $t=T, T-1, \ldots,2$,
\par \hspace*{0.4cm}For $j=1,\ldots,N_t$,
\par \hspace*{0.8cm}Compute an $\varepsilon_t^k$-optimal  solution $x_{t j}^k$ of
\begin{equation}\label{backisddp}
{\underline{\mathfrak{Q}}}_t^k (x_{t-1}^{k}, \xi_{t j}) = 
\left\{
\begin{array}{l}
\displaystyle \min_{x_t, z} \; f_t(x_t, z , \xi_{t j}) + \mathcal{Q}^{k}_{t+1}( x_t ) \\
A_{t j} x_t + B_{t j} z = b_{t j},\\
g_t(x_t, z , \xi_{t j}) \leq 0,\\
x_t \in \mathcal{X}_t,\\
z=x_{t-1}^k, \;\hspace*{4cm}[\lambda_{t j}^{k}]
\end{array}
\right.
\end{equation}
\par \hspace*{0.8cm}and an $\varepsilon_t^k$-optimal dual solution $\lambda_{t j}^{k}$ of the dual of problem 
\eqref{backisddp}
\par \hspace*{0.8cm}obtained dualizing constraints $z=x_{t-1}^k$.
\par \hspace*{0.4cm}End For
\par \hspace*{0.4cm}Compute 
$$
\begin{array}{l}
\beta_t^k = \sum_{j=1}^{N_t} p_{t j} \lambda_{t j}^k,\\
\theta_t^k = \sum_{j=1}^{N_t} p_{t j}\Big(   
f_t(x_{t j}^k , x_{t-1}^k , \xi_{t j} ) + \mathcal{Q}_{t+1}^k( x_{t j}^k ) - \langle \lambda_{t j}^k  , x_{t-1}^k \rangle 
\Big)
\end{array}
$$
\par \hspace*{0.4cm}and store the new cut 
\[
\mathcal{C}_t^k ( x_{t-1} ) := \theta_t^k -2 \varepsilon_t^k +   \langle \beta_t^k , x_{t-1}  \rangle 
\]
\par \hspace*{0.4cm}for $\mathcal{Q}_t$, making up the new approximation
$\mathcal{Q}_{t}^k = \max\{ \mathcal{Q}^{k-1}_{t}, \mathcal{C}_t^k  \}$.
\par End For
\item[Step 4)] Do $k \leftarrow k+1$ and go to Step 1).
\end{itemize}
\rule{\linewidth}{1pt}

\begin{remark} ISDDP algorithm given above applies both to differentiable and nondifferentiable
problems. In the differentiable case (when all functions $f_t(\cdot,\cdot,\xi_{t j})$
and $g_{t i}(\cdot,\cdot,\xi_{t j})$ are differentiable), compared to ISDDP introduced in \cite{guigues2016isddp},
the variant of ISDDP given above does not need to solve an additional optimization problem
to obtain the intercept of the cut. However, all subproblems solved in the forward and backward passes
have additional variables and constraints; the number of additional variables and constraints being
the size of $x_{t-1}$ for stage $t$.
\end{remark}

When objective functions $f_t(\cdot,\cdot,\xi_{t j})$
have saddle point representations (which is the case of all ``well structured'' convex functions), 
we can also derive another variant of ISDDP that combines SDDP with the inexact cuts given in
Section \ref{sec:icut2}.
For instance,
assuming to alleviate notation that $f_t$ is deterministic of the form $f_t(x_t, x_{t-1})$
with saddle point representation
\begin{equation}\label{fenchelt}
f_t(x_t, x_{t-1}) = x_{t-1}^T a_{t,1} + x_t^T a_{t,2} + \max_{w \in \mathcal{W}_t}  x_t^T {\bar A}_t w + x_{t-1}^T {\bar B}_t w - \Psi_t( w ),
\end{equation}
setting
$$
\begin{array}{l}
\Delta_{k+1} = \{\lambda=(\lambda_0,\lambda_1,\ldots,\lambda_k) \in \mathbb{R}^{k+1} : \lambda \geq 0, \sum_{i=0}^k \lambda_i =1  \},\\
{\bar \theta}_t^{0:k} = [\theta_t^0;\theta_t^1-2\varepsilon_t^1;\ldots;\theta_t^k-2\varepsilon_t^k],\;\beta_t^{0:k} = [\beta_t^0,\beta_t^1,\ldots,\beta_t^k],\\
\phi_{t k}( \lambda ) = - \lambda^T {\bar \theta}^{0:k},
\end{array}
$$
from the saddle point representation
$$
\mathcal{Q}_{t+1}^{k} ( x_t )  =   \max_{\lambda \in \Delta_{k+1}} \sum_{i=0}^k \lambda_i ( \theta_t^i -2\varepsilon_t^i + \langle  \beta_t^i , x_t \rangle     )
 =  \max_{\lambda_2 \in \Delta_{k+1}} x_t^T \beta_t^{0:k} \lambda_2 - \phi_{t,k} (\lambda_2),
$$
of $\mathcal{Q}_{t+1}^k$ where $\varepsilon_t^0=0$, we deduce the saddle point representation
\begin{equation}\label{fenchelbackward}
x_{t-1}^T a_{t,1} + x_t^T a_{t,2} + \max_{\lambda \in \Lambda} x_t^T \mathcal{A}_t^k \lambda + x_{t-1}^T \mathcal{B}_t \lambda - {\tilde \phi}_{t,k}(\lambda)
\end{equation}
of 
$f_t(x_t , x_{t-1}) + \mathcal{Q}_{t+1}^k ( x_t )$ where 
$$
\begin{array}{l}
\mathcal{A}_t^k =[{\bar A}_t, \beta_t^{0:k}],\;\mathcal{B}_t=[{\bar B}_t,0],\;{\tilde \phi}_{t,k} ( \lambda_1, \lambda_2 ) = \Psi_t (\lambda_1 ) + \phi_{t,k}( \lambda_2),\\ 
\Lambda = \{\lambda=(\lambda_1 , \lambda_2) : \lambda_1 \in \mathcal{W}_t, \lambda_2 \in \Delta_{k+1}   \}.
\end{array}
$$

\par In this situation, \eqref{fenchelbackward} provides a saddle point representation of 
the objective functions of problems \eqref{backisddp}  solved in the backward passes which allows us to build,
using Section \ref{sec:icut2}, inexact cuts of controlled accuracy
for value functions  ${\underline{\mathfrak{Q}}}_t^k(\cdot,\xi_{t j})$
and therefore for $\mathcal{Q}_t$.\\

\par We now study the convergence of ISDDP and start introducing more notation. Due to Assumption (H), the realizations of $(\xi_t)_{t=1}^T$ form a scenario tree of depth $T+1$
where the root node $n_0$ associated to a stage $0$ (with decision $x_0$ taken at that
node) has one child node $n_1$
associated to the first stage (with $\xi_1$ deterministic).
We denote by $\mathcal{N}$ the set of nodes and for a node $n$ of the tree, we define: 
\begin{itemize}
\item $C(n)$: the set of children nodes (the empty set for the leaves);
\item $x_n$: a decision taken at that node;
\item $p_n$: the transition probability from the parent node of $n$ to $n$;
\item $\xi_n$: the realization of process $(\xi_t)$ at node $n$:
for a node $n$ of stage $t$, this realization $\xi_n$ contains in particular the realizations
$b_n$ of $b_t$, $A_{n}$ of $A_{t}$, and $B_{n}$ of $B_{t}$.
\end{itemize}

Next, we define for iteration $k$ decisions $x_n^k$ for all node $n$ of the scenario tree
simulating the policy obtained in the end of iteration $k-1$ replacing 
cost-to-go function $\mathcal{Q}_t$ by 
$\mathcal{Q}_{t}^{k-1}$ for $t=2,\ldots,T+1$:\\
\rule{\linewidth}{1pt}
{\textbf{Simulation of ISDDP policy in the end of iteration $k-1$.}}\\
\rule{\linewidth}{1pt}
\hspace*{0.4cm}{\textbf{For }}$t=1,\ldots,T$,\\
\hspace*{0.8cm}{\textbf{For }}every node $n$ of stage $t-1$,\\
\hspace*{1.6cm}{\textbf{For }}every child node $m$ of node $n$, compute a $\delta_t^k$-optimal solution $x_m^k$ of
\begin{equation} \label{defxtkj}
{\underline{\mathfrak{Q}}}_t^{k-1}( x_n^k , \xi_m ) = \left\{
\begin{array}{l}
\displaystyle \inf_{x_m} \; f_t(x_m , x_n^k, \xi_m )  + \mathcal{Q}_{t+1}^{k-1}( x_m ) \\
A_{m} x_m + B_{m} x_n^k = b_{m},\\
g_t(x_m, x_n^k , \xi_{m}) \leq 0,\\
x_m \in \mathcal{X}_t,\\
\end{array}
\right.
\end{equation}
\hspace*{2.4cm}where $x_{n_0}^k = x_0$.\\
\hspace*{1.6cm}{\textbf{End For}}\\
\hspace*{0.8cm}{\textbf{End For}}\\
\hspace*{0.5cm}{\textbf{End For}}\\
\rule{\linewidth}{1pt}\\ 

We will assume that the sampling procedure in ISDDP satisfies the following property:\\

\par (H2) The samples in the backward passes are independent: $(\tilde \xi_2^k, \ldots, \tilde \xi_T^k)$ is a realization of
$\xi^k=(\xi_2^k, \ldots, \xi_T^k) \sim (\xi_2, \ldots,\xi_T)$ 
and $\xi^1, \xi^2,\ldots,$ are independent.\\

As said in the introduction, a useful tool for the convergence analysis of SDDP and ISDDP is Lemma 5.2 in \cite{lecphilgirar12} for vanishing errors
and Lemma 4.1 in \cite{guigues2016isddp} for bounded errors. 
We provide different proofs of these lemmas with slightly different assumptions, one of them
being stronger (the continuity of $f$ [which is satisfied when the lemmas are
applied to study the convergence of ISDDP]) and two being weaker.
More precisely, in these lemmas we do not assume $f^n\leq f$
and take equicontinuous sequences $f^n$ instead of sequences of Lipschitz continuous functions.
If we assumed $f^n \leq f$, the proof  would be a little shorter, because boundedness
of $\set{f^n}$ would be immediate. From these assumptions, we also derive a stronger conclusion, used
in the convergence analysis.
\begin{lemma} \label{technicallemma}
  Let $(X,d)$ be a compact metric space. If $\set{x_n}_{n\in\N}$ is a sequence in $X$,
  $\set{f^n}_{n\in\N}$ is an equicontinuous
  sequence of real functions on $X$,
  $f^1\leq f^2\leq f^3\leq\dots$, and
  $f$ is a continuous
  real function on $X$ then the
  following conditions are equivalent:
  \begin{itemize}
  \item[(a)] $\lim_{m,n\to\infty}f^m(x_n)-f(x_n)=0$.
  \item[(b)] $\lim_{n\to\infty}f^n(x_n)-f(x_n)=0$.
  \end{itemize}
 Morever, if (a) or (b) holds then $f^n$ converges uniformly to a continuous function 
 which coincides with $f$ on the set 
$$
    Y_*=\Set{y\in X\;:\;y=\lim_{j\to\infty}x_{n_j}
    \text{ for some subsequence }\set{x_{n_j}}_{j\in\N}}.
$$  
\end{lemma}
\begin{proof} See the Appendix. 
\end{proof}

The proof of the  previous lemma can be adapted to prove Lemma \ref{limsuptechlemma} which will be used in the convergence
analysis of ISDDP with bounded errors.

\begin{lemma}\label{limsuptechlemma} Let $(X,d)$ be a compact metric space, let 
$f: X \rightarrow \mathbb{R}$ be continuous and suppose that the sequence of equicontinuous functions 
$f^k, k \in \mathbb{N}$ satisfies $f^{k}(x) \leq f^{k+1}(x)\;\mbox{for all }x \in X,\;k \in \mathbb{N}$.
Let $(x^k)_{k \in \mathbb{N}}$ be a sequence in $X$ and  assume that 
\begin{equation}\label{limsuphyp}
\varlimsup_{k \rightarrow +\infty} f( x^k ) -f^k( x^k ) \leq S 
\end{equation}
for some finite $S \geq 0$.
Then
\begin{equation}\label{limsuphypkm1}
\varlimsup_{k \rightarrow +\infty} f( x^k ) -f^{k-1}( x^k ) \leq S. 
\end{equation}
Moreover, $f^n$ converges uniformly to a continuous function $g$ 
such that $|f(y)-g(y)|\leq S$ for every $y$ in the set
$$
    Y_*=\Set{y\in X\;:\;y=\lim_{j\to\infty}x_{n_j}
    \text{ for some subsequence }\set{x_{n_j}}_{j\in\N}}.
$$  
\end{lemma}
\begin{proof} See the Appendix. 
\end{proof}

We are now in a position to state our first convergence theorem for ISDDP.

\begin{theorem}[Convergence of ISDDP with bounded errors] \label{convisddplp}
Consider the sequences of decisions $(x_n^k)_{n \in \mathcal{N}}$ and of functions $(\mathcal{Q}_t^k)$ generated in the simulation of ISDDP.
Assume that (H), (H1), and (H2) hold, and that 
errors $\varepsilon_t^k$ and $\delta_t^k$ are bounded: $0 \leq \varepsilon_t^k \leq {\bar \varepsilon}$,
$0 \leq \delta_t^k \leq {\bar \delta}$ for finite ${\bar \delta}, {\bar \varepsilon}$. Then the following holds:
\begin{itemize}
\item[(i)] for $t=2,\ldots,T+1$, for all node $n$ of stage $t-1$,  almost surely
\begin{equation}\label{lowerbounds}
0 \leq \varliminf_{k \rightarrow +\infty} \mathcal{Q}_t( x_{n}^k ) - \mathcal{Q}_t^k( x_{n}^k ) \leq  
\varlimsup_{k \rightarrow +\infty} \mathcal{Q}_t( x_{n}^k ) - \mathcal{Q}_t^k( x_{n}^k ) \leq ({\bar \delta}  +  2{\bar{\varepsilon}})(T-t+1);
\end{equation}
\item[(ii)] for every $t=2,\ldots,T$, for all node $n$ of stage $t-1$,
the limit superior and limit inferior of the sequence of upper bounds  $\Big( \displaystyle  \sum_{m \in C(n)} p_m (  f_t( x_m^k, x_n^k, \xi_m )   + \mathcal{Q}_{t+1}( x_m^k ))  \Big)_k$ satisfy almost surely
{\small{
\begin{equation}\label{uppbounds}
\begin{array}{l}
0 \leq 
\varliminf_{k \rightarrow +\infty} \displaystyle \sum_{m \in C(n)} p_m \Big[ f_t( x_m^k, x_n^k, \xi_m )    + \mathcal{Q}_{t+1}( x_m^k ) \Big] - \mathcal{Q}_t( x_n^k ),  \\
\varlimsup_{k \rightarrow +\infty} \displaystyle  \sum_{m \in C(n)} p_m \Big[ f_t( x_m^k, x_n^k, \xi_m )    + \mathcal{Q}_{t+1}( x_m^k ) \Big] - \mathcal{Q}_t( x_n^k ) \leq ({\bar \delta} + 2{\bar{\varepsilon}})(T-t+1);
\end{array}
\end{equation}
}}
\item[(iii)]
the limit superior and limit inferior of the sequence   ${\underline{\mathfrak{Q}}}_1^{k-1}( x_{0} , \xi_1 )$ of lower bounds on the optimal value 
$\mathcal{Q}_1( x_0 )$ of \eqref{pbtosolve} satisfy almost surely
\begin{equation}\label{lbounds}
\mathcal{Q}_1( x_{0} )- {\bar{\delta}} T   - 2{\bar{\varepsilon}} (T-1) \leq \varliminf_{k \rightarrow +\infty} {\underline{\mathfrak{Q}}}_1^{k-1}( x_{0} , \xi_1 ) \leq  
\varlimsup_{k \rightarrow +\infty} {\underline{\mathfrak{Q}}}_1^{k-1}( x_{0} , \xi_1 ) \leq \mathcal{Q}_1( x_{0} );
 \end{equation}
\item[(iv)] for $t=2,\ldots,T$, almost surely the sequence of functions 
$(\mathcal{Q}_t^k)_k$ converges uniformly to a continuous function
$\mathcal{Q}_t^*$ which is at most at distance $({\bar \delta} + 2{\bar{\varepsilon}})(T-t+1)$  
from $\mathcal{Q}_t$
on every accumulation point ${\bar x}_n$ of the sequences $(x_n^k)_k$
for every node $n$ of stage $t-1$. 
\end{itemize}
\end{theorem}
\begin{proof} (i) We show \eqref{lowerbounds} for $t=2,\ldots,T+1$, and all node $n$ of stage $t-1$ by backward induction on $t$. The relation
holds for $t=T+1$. Now assume that it holds for $t+1$ for some $t \in \{2,\ldots,T\}$. Let us show that it holds for $t$.
Take a node $n$ of stage $t-1$.
Let $\mathcal{S}_n$ be the iterations where the sampled scenario passes through node $n$
and take an iteration $k \in \mathcal{S}_n$.
It was shown in Lemma 5.2 in \cite{guigues2016isddp} that for the classes of 
problems we consider, Assumptions (H1)-3),5) imply that almost surely for every $j, k$, there exists 
$x_t$ satisfying
$$
x_t \in \mbox{ri}(\mathcal{X}_t),
A_{t j}x_t + B_{t j}x_{t-1}^k = b_{t j} \mbox { and }g_t(x_t,x_{t-1}^k,\xi_{t j})<0.
$$
Recalling that $\mathcal{X}_t \times \mathcal{X}_{t-1} \subset \mbox{dom}(g_{t i})$ for all $i$, we can reproduce the
reasoning used just after Proposition \ref{varprop1} in Section \ref{sec:comparison} to deduce that 
for every $j,t$, there exists 
\begin{equation}\label{slaterbackstk}
(x_t, z) \in \mbox{ri}(S_{t j})
\end{equation}
where 
$$
S_{t j}=\{(x_t,z):A_{t j} x_t + B_{t j} z = b_{t j},g_t(x_t, z , \xi_{t j}) \leq 0,x_t \in \mathcal{X}_t\}.
$$
Condition \eqref{slaterbackstk} is exactly Slater condition \eqref{slater} (from Proposition \ref{icutsec2})
written for problem \eqref{backisddp} solved in the backward pass of iteration $k$ for scenario $j$. 
Therefore, we can apply Proposition \ref{icutsec2} to value function 
${\underline{\mathfrak{Q}}}_t^k (\cdot, \xi_{t j})$ to obtain 
a $2\varepsilon_t^k$-inexact cut for this function for stage $t$ and iteration $k$ of ISDDP.
More precisely, fix $j\in \{1,\ldots,N_t\}$ and take $m$
such that $\xi_{t j}=\xi_m$.
Recalling that $\lambda_m^k$ is defined in \eqref{defxtkj} and setting
$$
\mathcal{C}_{t m}^k (  x_n ) =
f_t(x_m^k , x_n^k  ,  \xi_m   ) + \mathcal{Q}_{t+1}^k ( x_n^k ) -2\varepsilon_t^k + \langle \lambda_m^k, x_n - x_n^k  \rangle,
$$
using Proposition \ref{icutsec2}, we get for all $x_n \in \mathcal{X}_{t-1}$ and $k \in \mathcal{S}_n$:
\begin{equation}\label{cutisddp1}
\mathcal{C}_{t m}^k (  x_n ) \leq {\underline{\mathfrak{Q}}}_t^k(  x_n ,   \xi_m  ) 
\end{equation}
and 
\begin{equation}\label{cutisddp2}
{\underline{\mathfrak{Q}}}_t^k(  x_n^k ,   \xi_m  ) - \mathcal{C}_{t m}^k (  x_n^k ) \leq 2 \varepsilon_t^k.  
\end{equation}
This implies that $\mathcal{Q}_t^k$ is indeed a valid cut for $\mathcal{Q}_t$: for $x_n \in \mathcal{X}_{t-1}$
and $k \in \mathcal{S}_n$, we have
\begin{equation}\label{validcut}
\begin{array}{lcl}
\mathcal{Q}_t( x_ n) 
=\sum_{m \in C(n)} p_m \mathfrak{Q}_t(x_n , \xi_m   )
& \geq &  \sum_{m \in C(n)} p_m {\underline{\mathfrak{Q}}}_t^k(x_n , \xi_m   ) \\
& \stackrel{\eqref{cutisddp1}}{\geq } &
 \sum_{m \in C(n)} p_m \mathcal{C}_{t m}^k (  x_n ) = \mathcal{C}_{t}^k (  x_n ).
\end{array}
 \end{equation}
Also by definition of $x_m^k$ computed in the simulation of iteration $k$ we get
\begin{equation}\label{cutisddp3}
f_t( x_m^k , x_n^k , \xi_m  ) + \mathcal{Q}_{t+1}^{k-1}( x_m^k ) 
\leq {\underline{\mathfrak{Q}}}_t^{k-1}(  x_n^k ,   \xi_m  ) +  \delta_t^k.  
\end{equation}
Therefore, for $k \in \mathcal{S}_n$:
\begin{equation}\label{cutctk}
\begin{array}{lcl}
\mathcal{C}_t^k (  x_n^k   ) &   =   & \displaystyle \sum_{m \in C(n)}
p_m \mathcal{C}_{t m}^k( x_n^k ),\\
&   \stackrel{\eqref{cutisddp2}}{\geq}  & \displaystyle \sum_{m \in C(n)} p_m \left[ {\underline{\mathfrak{Q}}}_t^k(  x_n^k ,   \xi_m  )
- 2 \varepsilon_t^k \right],\\
&  \geq &  - 2 {\bar \varepsilon} +  \displaystyle \sum_{m \in C(n)} p_m  {\underline{\mathfrak{Q}}}_t^{k-1}(  x_n^k ,   \xi_m  ),\\
&  \stackrel{\eqref{cutisddp3}}{\geq} &  - 2 {\bar \varepsilon} +  \displaystyle \sum_{m \in C(n)} p_m  \left[f_t( x_m^k , x_n^k , \xi_m  ) + \mathcal{Q}_{t+1}^{k-1}( x_m^k ) - \delta_t^k \right],\\
&  \geq  &  - 2 {\bar \varepsilon} - {\bar \delta} +  \displaystyle \sum_{m \in C(n)} p_m  \left[f_t( x_m^k , x_n^k , \xi_m  ) + \mathcal{Q}_{t+1}^{k-1}( x_m^k ) \right].
\end{array}
\end{equation}
It follows that for $k \in \mathcal{S}_n$
\begin{equation}\label{eqproofconvisddp1}
\begin{array}{lcl}
0 &  \stackrel{\eqref{validcut}}{\leq} & \mathcal{Q}_t( x_n^k ) 
- \mathcal{Q}_t^k(x_n^k) 
\leq \mathcal{Q}_t( x_n^k ) - \mathcal{C}_t^k( x_n^k ) \\
 &  \stackrel{\eqref{cutctk}}  {\leq} &  2 {\bar \varepsilon} + {\bar \delta} +  \displaystyle \sum_{m \in C(n)} p_m  \left[
\mathfrak{Q}_t( x_n^k , \xi_m ) - f_t( x_m^k , x_n^k , \xi_m  ) - \mathcal{Q}_{t+1}^{k-1}( x_m^k ) \right]\\
& \leq & \displaystyle 2 {\bar \varepsilon} + {\bar \delta} +   \sum_{m \in C(n)} p_m  \Big[ \underbrace{\mathfrak{Q}_t(x_n^k , \xi_m ) -  f_t( x_m^k , x_n^k , \xi_m  )  - \mathcal{Q}_{t+1}( x_m^k ) }_{\leq 0\mbox{ by definition of }\mathfrak{Q}_t\mbox{ and }x_m^k} \Big] \\
& &  + \displaystyle  \sum_{m \in C(n)} p_m  \Big[ \mathcal{Q}_{t+1}( x_m^k )  -    \mathcal{Q}_{t+1}^{k-1}( x_m^k ) \Big].
\end{array}
\end{equation}
Using the induction hypothesis, we have for every $m \in C(n)$ that
$
\varlimsup_{k \rightarrow +\infty} \mathcal{Q}_{t+1}( x_{m}^k ) - \mathcal{Q}_{t+1}^k( x_{m}^k ) \leq ({\bar \delta}  +  2{\bar{\varepsilon}})(T-t).
$
Following the proof of Lemma 4.2 in \cite{guiguesmonteiro2019}, we obtain that sequence $(\beta_t^k)_k$ is almost surely bounded
and that functions $(\mathcal{Q}_t^k)_k$ are $L$-Lipschitz continuous
and therefore sequence  $(\mathcal{Q}_t^k)_k$ is monotone and equicontinuous.
Since $\mathcal{Q}_t$ is continuous on $\mathcal{X}_{t-1}$,  we can apply Lemma \ref{limsuptechlemma} to obtain
$
\varlimsup_{k \rightarrow +\infty} \mathcal{Q}_{t+1}( x_{m}^k ) - \mathcal{Q}_{t+1}^{k-1}( x_{m}^k ) \leq ({\bar \delta}  +  2{\bar{\varepsilon}})(T-t),
$
which, plugged into \eqref{eqproofconvisddp1}, gives
\begin{equation}\label{proofSn}
\varlimsup_{k \rightarrow +\infty, k \in \mathcal{S}_n} \mathcal{Q}_t( x_n^k ) - \mathcal{Q}_t^k ( x_n^k ) \leq ({\bar \delta}  + 2 {\bar{\varepsilon}})(T-t+1).
\end{equation}
Finally, to conclude the proof of (i), it remains to show that
\begin{equation}\label{proofnotSn}
\varlimsup_{k \rightarrow +\infty, k \notin \mathcal{S}_n} \mathcal{Q}_t( x_n^k ) - \mathcal{Q}_t^k ( x_n^k ) \leq ({\bar \delta}  + 2 {\bar{\varepsilon}})(T-t+1),
\end{equation}
and with relation \eqref{proofSn} at hand, relation \eqref{proofnotSn} can be shown by contradiction following the 
end of the proof of Theorem 4.2 in \cite{guigues2016isddp}.
\par (ii) and (iii) can be shown using (i) and following the proof of  Theorem 4.2-(ii), (iii) in \cite{guigues2016isddp}.
\par (iv) is an immediate consequence of (i) and Lemma \ref{limsuptechlemma}.
\end{proof}

We can now state our second convergence theorem for ISDDP:

\begin{theorem}[Convergence of ISDDP with vanishing errors] \label{convisddplp2}
Consider the sequences of decisions $(x_n^k)_{n \in \mathcal{N}}$ and of functions $(\mathcal{Q}_t^k)$ generated 
in the simulation of ISDDP.
Assume that (H), (H1), and (H2) hold, and that 
for all $t$ we have
$\lim_{k \rightarrow +\infty}\varepsilon_t^k = \lim_{k \rightarrow +\infty}\delta_t^k = 0$. 
Then almost surely the limit of the sequence $({\underline{\mathfrak{Q}}}_1^{k-1}( x_{0} , \xi_1 ))_{k \geq 1}$ 
is the optimal value 
$\mathcal{Q}_1( x_0 )$ of \eqref{pbtosolve}.
Moreover, for $t=2,\ldots,T$, almost surely the sequence of functions 
$(\mathcal{Q}_t^k)_k$ converges uniformly to a continuous function
$\mathcal{Q}_t^*$ which coincides with $\mathcal{Q}_t$
on every accumulation point ${\bar x}_n$ of the sequences $(x_n^k)_k$
for every node $n$ of stage $t-1$. 
\end{theorem}
\begin{proof} It suffices to follow the proof of Theorem \ref{convisddplp} and to use 
Lemma \ref{technicallemma} instead of Lemma \ref{limsuptechlemma}. 
\end{proof}

If instead of the inexact cuts from Section \ref{sec:icut1} we use in ISDDP the inexact cuts from Section 
\ref{sec:icut2} based
on saddle point representations of the objective, we obtain similar convergence results, due to  the fact
that the error terms in both the cuts from Section \ref{sec:icut1} and from Section \ref{sec:icut2} linearly depend on 
$\delta_t^k$ and $\varepsilon_t^k$.

\section{Numerical experiments}\label{sec:numexp}

We consider the multistage nondifferentiable nonlinear  stochastic 
program given by the following DP equations:
the Bellman function for stage 
$t=1,\ldots,T$, is
$\mathcal{Q}_t(x_{t-1})=\mathbb{E}_{\xi_t,\Psi_t,U_t}
[\mathfrak{Q}_t(x_{t-1},\xi_t,\Psi_t,U_t)]$ and for
$t=1,\ldots,T$, $\mathfrak{Q}_t(x_{t-1},\xi_t,\Psi_t,U_t)$ is given by
\begin{equation}\label{pbsim}
\begin{array}{l}
\min \;f_t(x_t,x_{t-1},\xi_t,U_t) + \mathcal{Q}_{t+1}(x_t)\\
-100\,{\textbf{e}} \leq x_t \leq 100\,{\textbf{e}},\\
\max(  4(x_t - {\textbf{e}})^T (x_t-{\textbf{e}}), 
x_t^T (\xi_t \xi_t^T +\alpha I_n )x_t + x_t^T \xi_t + 1 ) \leq \Psi_t,
\end{array}
\end{equation}
where $x_t \in \mathbb{R}^n$,
$
f_t(x_t,x_{t-1},\xi_t,U_t)=
\max( (x_t-x_{t-1})^T (\xi_t \xi_t^T +\alpha I_n )(x_t-x_{t-1})
+ x_t^T \xi_t + 1, x_t^T (\xi_t \xi_t^T +\alpha I_n ) x_t + x_t^T {\textbf{e}} + U_t ),
$ {\textbf{e}} is a vector of size $n$ of
ones, and $\mathcal{Q}_{T+1}$ is the null function. In these equations, $\alpha \geq 0$ is a parameter, 
$\xi_t$ is a discretization of a Gaussian random vector
with mean vector $m_t$ having entries $1$ or $-1$ and covariance matrix $\Sigma_t=A_t A_t^T + 0.5 I$
where $A_t$ has entries in $[-0.5,0.5]$; $U_t$ is a discrete random variable taking values $+10$, $-10$,
and $\Psi_t$ has discrete distribution with support contained
in $[10^4,10^5]$. The number of realizations $N_t$ for $(\xi_t,\Psi_t,U_t)$
is fixed to $N_t=N$ for each stage. We assume
that 
$(\xi_1,\Psi_1,U_1)$ is known
and $(\xi_2,\Psi_2,U_2),\ldots,(\xi_T,\Psi_T,U_T)$ are independent.

We generate 2 instances of this problem with parameters
$\alpha=0.2$ and
$T,n,M$ given by   
$(T,n,M)=(5,10,20)$
and $(T,n,M)=(5,50,20)$. The instances are chosen taking realizations 
$\Psi_{t j}$ of $\Psi_{t}$ sufficiently large, in such a way that  
Assumption (H1)-4) holds. It is easy to check that the remaining assumptions
(H1) and (H) are satisfied and therefore SDDP
can be applied to solve the problem as well
as SDDP combined with the inexact cuts from Section
\ref{sec:icut1}. In what follows, we denote the corresponding solution
methods by {\tt{SDDP}} and {\tt{ISDDPND}}
(Inexact SDDP for nondifferentiable problems).
We also solved problem \eqref{pbsim}
using Stochastic Dynamic Cutting Plane
(denoted by {\tt{StoDCuP}}), StoDCuP combined with inexact cuts (denoted by
 {\tt{IStoDCuP}}) introduced
in \cite{guiguesmonteiro2019} as well
as the inexact variant of SDDP introduced in
\cite{guigues2016isddp} that we will denote by
{\tt{ISDDPD}} (Inexact SDDP for differentiable problems) in what follows (the interested reader can find
in the Appendix
the formulas for the inexact cuts to use for this
inexact variant of SDDP).
Observe that this inexact variant {\tt{ISDDPD}} was designed
for differentiable problems but can be applied
to \eqref{pbsim} 
reformulating the problem 
as a differentiable problem replacing in
\eqref{pbsim} each max with 2 quadratic constraints.
Finally, we consider a
mixed StoDCuP-SDDP variant (denoted by {\tt{MSDDP}})
which uses StoDCuP for the first 150 iterations and
SDDP for the remaining iterations, as well as its
inexact counterpart (denoted by {\tt{IMSDDP}})
which is StoDCuP with inexact cuts, i.e., {\tt{IStoDCuP}}, for the first 150
iterations and SDDP with the inexact cuts
from Section \ref{sec:icut1}, i.e., {\tt{ISDDPND}}, for the remaining iterations.  The Matlab implementation
of all methods can be found at 
\url{https://github.com/vguigues/ISDDP_NLP}. 
All subproblems were solved using Mosek
optimization library \cite{mosek}.

%


For the inexact variants with inexact cuts to be well defined, we also need
to set the level of accuracy of the computed
solutions along the iterations of the methods.
In our experiments, the
relative error of the subproblem solutions (Mosek
parameter MSK\_DPAR\_INTPNT\_TOL\_REL\_GAP
 whose range is any value $\geq 10^{-14}$ and default value is
$10^{-8}$) is given in Table \ref{tableacc}; see also
Remark 2 in \cite{guigues2016isddp} for other choices of sequences
of noises $\varepsilon_t^k$. For the exact variants, this parameter
was set to $10^{-10}$ for all iterations.

\begin{table}
\centering
{\small{
\begin{tabular}{|c|c|c|c|c|c|c|c|}
\hline
 Iteration   & 1-10  & 11-20    & 21-40   & 41-140
&  141-240  & 241-350 & $>350$ \\
\hline
Parameter value   & 10  & 5    & 3   &1 
&   0.5 & 0.1& e-6\\
\hline
 \end{tabular}
}}
\caption{Relative error of the subproblem solutions
along iterations of inexact methods (Mosek parameter MSK\_DPAR\_INTPNT\_TOL\_REL\_GAP).}\label{tableacc}
\end{table}

All methods
compute at each iteration
a lower bound on the optimal value which
is the optimal value of the first stage
problem solved in the forward pass
and upper bounds computed by Monte-Carlo
simulations, from iteration 400 on, using 
the last 400 forward scenarios. 
The algorithms stopped when a relative
gap of at most 0.1 was achieved
or, for the largest instance, when
the maximal number of 600 iterations
was reached.

The number of iterations before
stopping the algorithms as well
as the CPU time is given
in Table \ref{tablevarcpu} for all methods and
the two instances. 

\begin{table}
\centering
\begin{tabular}{|c|c|c|c|c|c|c|c|}
\hline
 & {\tt{IMSDDP}}   & 
 {\tt{ISDDPND}}  & {\tt{ISDDPD}} & {\tt{IStoDCuP}}    & 
{\tt{MSDDP}}  & {\tt{SDDP}} &
{\tt{StoDCuP}} \\
\hline
Iterations   &439&409 & 465 &655  &569&431&770\\
\hline
CPU time   &233.1&282.2& 322.5 &
 582.4&352.7&297.3&791.8\\
\hline
\end{tabular}
\begin{center}
$T=5, n=10, M=20$
\end{center}
\begin{tabular}{|c|c|c|c|c|c|c|c|}
\hline
 & {\tt{IMSDDP}}   & 
 {\tt{ISDDPND}}  &{\tt{ISDDPD}} & {\tt{IStoDCuP}}    & 
{\tt{MSDDP}}  & {\tt{SDDP}} &
{\tt{StoDCuP}} \\
\hline
Iterations   &400&400 & 400  & -   &400 &400&-\\
\hline
CPU time   &3 424& 4 387 & 3 237&   - &3 547&4 504&-\\
\hline
\end{tabular}
\begin{center}
$T=5, n=50, M=20, \alpha=0.2$
\end{center}
\caption{Number of iterations and CPU time in seconds needed to solve the two instances. 
For the second instance, the unfilled cells for
{\tt{IStoDCuP}} and {\tt{StoDCuP}}
indicate that these methods
had not converged after completing
the maximal number of 600 iterations.
It took {\tt{IStoDCuP}} (resp. {\tt{StoDCuP}})
2 230 s. (resp. 2 356 s.)
to complete these 600 iterations.}
\label{tablevarcpu}
\end{table}

The evolution of the upper and lower
bounds for some iterations,
all methods, and the two instances
is given in Tables 
\ref{tablelower} and \ref{tableupper}.

We observe that the sequences of upper bounds
decrease and as expected the sequences of lower bounds are
increasing
and both sequences
 converge to the same values. On these instances, {\tt{StoDCuP}}
and its inexact variant 
{\tt{IStoDCuP}} need more iterations
and time 
than the other methods to converge
(for the largest instance the maximal
number of 600 iterations was
even not enough for {\tt{StoDCuP}}
and {\tt{IStoDCuP}} to converge).
However, we observed that the first
iterations of {\tt{StoDCuP}}
and {\tt{IStoDCuP}} are much quicker
than the first iterations of {\tt{SDDP}}
and its inexact variants, which explains
the good performance of the mixed
StoDCuP-SDDP method and its inexact
counterpart. Indeed, 
{\tt{IMSDDP}} is the quickest to converge
for the first instance and the second
quickest, after {\tt{ISDDPD}}, for the largest
instance. In particular, both {\tt{MSDDP}}
and {\tt{IMSDDP}} converge
much quicker than {\tt{SDDP}}.
Out of the 8 runs of the inexact methods, only
one did not converge quicker than its
exact counterpart, namely {\tt{ISSDPD}}
for the smallest instance.
Among inexact variants {\tt{ISSDPD}}
and {\tt{ISSDPND}}  of {\tt{SDDP}},
method {\tt{ISSDPD}} was the quickest 
on the instance with
the largest value of the state vector
size $n$ ($n=50$) while
{\tt{ISSDPND}} was the quickest on
the smallest instance, which may come
from the increase in the CPU time
needed to solve subproblems
with {\tt{ISSDPND}} due to the copy
of variables used to derive the cuts.
On the other hand, {\tt{ISSDPND}}
is more general and can apply to
nondifferentiable problems contrary
to {\tt{ISSDPD}}.

\begin{table}
\centering
{\small{
\begin{tabular}{|c|c|c|c|c|c|c|c|}
\hline
{\tt{Iteration}}  & {\tt{IMSDDP}}   & 
 {\tt{ISDDPND}}  & {\tt{ISDDPD}} & {\tt{IStoDCuP}}    & 
{\tt{MSDDP}}  & {\tt{SDDP}} &
{\tt{StoDCuP}} \\
\hline
 400  &14.34&14.66& 14.32 & 5.07   &14.35&14.66&2.76\\
\hline
 409  &14.39&14.66& 14.41   &6.07   &14.41&14.66&4.67\\
\hline
 431  &14.46&-&14.45   &9.17   &14.46&14.67&7.47\\
\hline
 439  &14.48&-& 14.49 & 9.45  &14.49&-&8.95\\
\hline
 465  &-&-& 14.62  &12.80   &14.57&-&12.34\\
\hline
 500  &-&-& -  &12.80   &14.57&-&12.34\\
\hline
 569  & -&-& - & 13.71  &14.62&-&13.57\\
\hline
 770  &- &- & - & - & -& -& 13.97\\
\hline
\end{tabular}
}}
\begin{center}
$T=5, n=10, M=20, \alpha=0.2$
\end{center}

{\small{
\begin{tabular}{|c|c|c|c|c|c|c|c|}
\hline
{\tt{Iteration}}  & {\tt{IMSDDP}}   & 
 {\tt{ISDDPND}} & {\tt{ISDDPD}} & {\tt{IStoDCuP}}    & 
{\tt{MSDDP}}  & {\tt{SDDP}} &
{\tt{StoDCuP}} \\
\hline
 200  & -96 077  & 84.8 & 84.4 &-1.832e6  &-8 884&83.8&-1.843e6\\
\hline
 300  & 53.7 & 85.8 & 85.7  &-1.05e6  & 35.1 & 85.6&-1.0e6  \\
\hline
 400  & 84.6 & 85.9  & 85.9  & -6.6e5   & 84.5 & 85.9 & -7.2e5\\
\hline
 600  & - & -  &  - & -3.3e4   & - & - & -3.5e4\\
\hline
\end{tabular}
}}
\begin{center}
$T=5, n=50, M=20, \alpha=0.2$
\end{center}
\caption{Lower bounds computed along the iterations
of the methods for both instances.}\label{tablelower}
\end{table}

\begin{table}
\centering
{\small{
\begin{tabular}{|c|c|c|c|c|c|c|c|}
\hline
{\tt{Iteration}}  & {\tt{IMSDDP}}   & 
 {\tt{ISDDPND}}  & {\tt{ISDDPD}} & {\tt{IStoDCuP}}    & 
{\tt{MSDDP}}  & {\tt{SDDP}} &
{\tt{StoDCuP}} \\
\hline
 400  &20.81 &   17.79 & 20.4 &32.61 & 22.17 &  
 20.55 &   43.19  \\
\hline
 409  &19.03&15.72& 18.3&27.47  &21.94&19.79&26.32\\
\hline
 431  &16.48&-&17.1 &19.85  &18.57&16.25&16.25\\
\hline
 439 &15.89 &-&16.81 &18.78  &18.56&-& 20.03\\
\hline
 465 &- &-&15.9 &18.42  &18.14&-& 19.37\\
\hline
 500  &- &- & - &17.11   &16.75 & - &17.72\\
\hline
 569 &-&-&-  &16.42   &16.22&-&16.86\\
\hline
 770 &-&-& -&  &-&-&15.94\\
\hline
\end{tabular}
}}
\begin{center}
$T=5, n=10, M=20, \alpha=0.2$
\end{center}

{\small{
\begin{tabular}{|c|c|c|c|c|c|c|c|}
\hline
{\tt{Iteration}}  & {\tt{IMSDDP}}   & 
 {\tt{ISDDPND}} & {\tt{ISDDPD}} & {\tt{IStoDCuP}}    & 
{\tt{MSDDP}}  & {\tt{SDDP}} &
{\tt{StoDCuP}} \\
\hline
 400  & 86.22 &  86.7  & 86.1 &  21 348 & 87.7 & 89.0  &   19 538 \\
\hline
 600  & - &  -  & - &9 342 & - & -  &   7 231 \\
\hline
\end{tabular}
}}
\begin{center}
$T=5, n=50, M=20, \alpha=0.2$
\end{center}
\caption{Upper bounds computed along the iterations
of the methods for both instances..}\label{tableupper}
\end{table}

\section{Conclusion}

In \cite{guigues2016isddp}, an inexact variant of SDDP called ISDDP was introduced. Two variants of 
the method were described in \cite{guigues2016isddp}: one for linear problems and one for nonlinear differentiable problems. In this paper, we explained how to extend ISDDP for nondifferentiable multistage stochastic programs.
We provided formulas to compute inexact cuts for value functions of possibly nondifferentiable 
optimization problems and combined these cuts with SDDP to describe 
two new inexact variants of SDDP, one for each of the classes of cuts derived (the cuts from Section 
\ref{sec:icut1} and the cuts from Section \ref{sec:icut2}).

Several comments are in order:
\begin{itemize}
\item the variants of ISDDP presented in this paper can be used both for nonlinear differentiable
and nonlinear nondifferentiable optimization problems.
\item For errors bounded from above by $\varepsilon$, same as ISDDP for linear programs introduced in \cite{guigues2016isddp}, ISDDP variants of this paper provide $3\varepsilon T$-optimal first stage 
solutions. Using the analysis of Section \ref{sec:comparison}, it is easy to check that 
ISDDP for nonlinear stochastic programs from \cite{guigues2016isddp} provides for bounded errors 
a $O(T \sqrt{\varepsilon})$-optimal first stage solution. 
However, all subproblems solved in the forward and backward passes of the variant of ISDDP
that uses the cuts from Section \ref{sec:icut1}
have additional variables and constraints; the number of additional variables and constraints being
the size of $x_{t-1}$ for stage $t$.
\item All variants of ISDDP from \cite{guigues2016isddp} and from this paper converge
to an optimal policy for vanishing noises. The convergence analysis of ISDDP applied 
to nonlinear programs in \cite{guigues2016isddp} was however more technical due to 
the fact that the error terms in the inexact cuts were not a linear function of $\delta_t^k$
and $\varepsilon_t^k$ (see Proposition 5.4 in \cite{guigues2016isddp}).
\end{itemize}

\section{Appendix}

\begin{lemma}\label{lemmasumcocoercive} Assume that $F_i:\mathbb{R}^m \rightarrow \mathbb{R}^m$ is $L_i$-co-coercive for 
$i=1,\ldots,n$. Then $\sum_{i=1}^n F_i$ is $(\sum_{i=1}^n L_i)$-co-coercive. 
\end{lemma}
\begin{proof} We can assume w.l.o.g that all  $L_i$ are positive.
Let $S(x)=\sum_{i=1}^n F_i(x)$, $L=\sum_{i=1}^n L_i>0$ , and $\alpha_i=\frac{L_i}{L}$. 
Observing that $\sum_{i=1}^n \alpha_i=1$ and
using the convexity of $\|\cdot\|^2$ we get:
\begin{equation}
\begin{array}{lcl}
\langle y-x, S(y)-S(x)\rangle & \geq & \sum_{i=1}^n \frac{1}{L_i} \|F_i(x) - F_i(y)\|^2 \\
&  = & \frac{1}{L} \sum_{i=1}^n \alpha_i \| \frac{1}{\alpha_i} ( F_i(x) - F_i(y) )\|^2 \\
&  \geq & \frac{1}{L} \|S(y)-S(x)\|^2,
\end{array}
\end{equation}
which achieves the proof of the lemma.
\end{proof}

\par {\textbf{Proof of Lemma \ref{technicallemma}.}}
  Implication (a)$\Rightarrow$(b) holds trivially.
  Suppose (b) holds.
  Since $X$ is compact and $f$ is continuous, the sequence
  $\set{f^n(x_n)}$ is bounded. Combining this result with the
  compactness of  $X$ and the equicontinuity of $\set{f^n}$ we conclude
  that this sequence is pointwise uniformly bounded. Hence the
  monotone sequence $\set{f^n(x)}$ converges for any $x\in X$.  
  Recall that $Y_*$ is the set of limit points of $\set{x_n}$
  and let $g:X\to\R$ be the pointwise limit of $\set{f^n}$ that is,
  $$
  \begin{array}{l}
    g(x)=\lim_{n\to\infty}f^n(x)\;\;\;(x\in X).\\
  \end{array}
 $$ 
  We claim that
  \begin{enumerate}
  \item $g$ is  continuous;
  \item $\set{f^n}$ converges uniformly to $g$;
  \item $g(y)=f(y)$ for any $y\in Y_*$.
  \end{enumerate}
  Continuity of $g$ follows from the
  equicontinuity of $\set{f^n}$ and its convergence to $g$. Since
  $\set{f^n}$ is a sequence of equicontinuous functions converging
  monotonically in a compact set to a continuous function $g$, this
  convergence is uniform. To prove item 3, suppose that
  $\lim_{j\to\infty} x_{n_j}=y$.  Direct use of the triangle inequality
  yields
  \begin{align*}
    \abs{f^{n_j}(y)-f(y)} \leq
    \abs{f^{n_j}(y)-f^{n_j}(x_{n_j})}+
    \abs{f^{n_j}(x_{n_j})-f(x_{n_j})}+
    \abs{f(x_{n_j})-f(y)}.
  \end{align*}
  It follows from the equicontinuity of $\set{f^n}$, the continuity of
  $f$, and the convergence of $\set{x_{n_j}}$ to $y$ that the first and
  third terms in the right-hand side of the above inequality converge
  to $0$, while it follows from Assumption (b) that the
  middle term also converges to $0$. Since $\set{f^{n_j}(y)}$
  converges to $g(y)$, we have $g(y)=f(y)$.
  
  To end the proof, take $\varepsilon>0$.
  There exists $M_0\in\N$ such that
  \begin{align*}
    m\geq M_0\Rightarrow \abs{f^m(x)-g(x)}<\varepsilon \qquad \forall x\in X.
  \end{align*}
  It follows from the continuity of $f$ and $g$, and from
  the compactness of $X$
  that there is $\delta>0$ such
  that
  \begin{align*}
    d(x,x')\leq \delta\Rightarrow
    \;\abs{f(x)-f(x')}\leq\varepsilon,
    \;\abs{g(x)-g(x')}\leq\varepsilon.
  \end{align*}
  It follows from the definition of $Y_*$ and the compactness of $X$ that
  there is $N_0\in\N$ such that $d(x^n,Y_*)<\delta$ for $n\geq N_0$.
  Suppose that $m\geq M_0$ and $n\geq N_0$. There is $y\in Y_*$ such
  that $d(x^n,y)<\delta$.
  Therefore
  \begin{equation}\label{lemmalasteq}
 \begin{array}{lcl}  
    \abs{f^m(x_n)-f(x_n)}
    & \leq &\abs{f^m(x_n)-g(x_n)}+\abs{g(x_n)-g(y)}
      +\abs{g(y)-f(x_n)}\\
      & =&\abs{f^m(x_n)-g(x_n)}+\abs{g(x_n)-g(y)}
        +\abs{f(y)-f(x_n)}<3\varepsilon,
  \end{array}
  \end{equation}
  which achieves the proof of the lemma.$\hfill \square$\\

\par {\textbf{Proof of Lemma \ref{limsuptechlemma}.}} 
The proof is a simple extension of the proof of Lemma \ref{technicallemma}.
We outline the changes in the proof below.
Since the sequence  $f^{n}(x_n)-f(x_n)$ is bounded from above and $f$ is continuous on the compact
set $X$, the sequence $f^n(x_n)$ is bounded from above. Same as in Lemma \ref{technicallemma},
together with the equicontinuity, the monotonicity of $f^n$, and the compactness of $X$, this
implies that the sequence $f^n(x)$ converges for every $x \in X$ uniformly
to a continuous function $g$. 
For every $y \in Y_*$,
taking $\{x_{n_j}\}$ satisfying $y = \lim_{j \rightarrow \infty} x_{n_j}$, we get 
$$
\begin{array}{lll}
|g(y)-f(y)|&=&|\lim_{j \rightarrow \infty} f^{n_j}(y) - f(\lim_{j \rightarrow \infty} x_{n_j})|
= |\lim_{j \rightarrow \infty} f^{n_j}(y) - f( x_{n_j})| \\
& \leq  & | \lim_{j \rightarrow \infty} f^{n_j}(y) - f^{n_j}(x_{n_j})|  + | \lim_{j \rightarrow \infty} f^{n_j}(x_{n_j}) - f(x_{n_j})|=S. 
\end{array}
$$
To conclude, it suffices to modify the last inequality \eqref{lemmalasteq} in Lemma \ref{technicallemma} by 
$$
 \begin{array}{l}  
    \abs{f^m(x_n)-f(x_n)}
     \leq \abs{f^m(x_n)-g(x_n)}+\abs{g(x_n)-g(y)}
      +\abs{g(y)-f(x_n)}
    \\
       \leq \abs{f^m(x_n)-g(x_n)}+\abs{g(x_n)-g(y)}
        +\abs{g(y)-f(y)}+\abs{f(y)-f(x_n)}\\
        \leq  S+ 3\varepsilon,
  \end{array}
$$
which concludes the proof of the lemma.$\hfill \square$\\

\par {\textbf{Formulas for inexact cuts for ISDDP from
\cite{guigues2016isddp} applied to problem  \eqref{pbsim}.}} The inexact cut for ISDDP from
\cite{guigues2016isddp} applied to problem  \eqref{pbsim} for $\mathcal{Q}_t$
takes the form 
$\mathcal{C}_t^k(x_{t-1})=
\theta_t^k - \eta_t^k 
+ \langle \beta_t^k , x_{t-1} \rangle
$
for iteration $k$.
This cut is computed as follows.
Given trial point $x_{t-1}^k$
we compute for $j=1,\ldots,N_t$, an
approximate optimal 
primal-dual
solution $(f_{tj}^*, q_{tj}^*, x_{t j}^{*},\lambda_{1 j}^*)$
of 
\begin{equation}\label{pbsimsddp}
\begin{array}{l}
\min_{f,q,x_t} \;f + q\\
f \geq (x_t-x_{t-1}^k)^T (\xi_{t j} \xi_{t j}^T +\alpha I_n )(x_t-x_{t-1}^k) + x_t^T \xi_{t j} + 1,\;\;[\lambda_{1 j}]\\
f \geq x_t^T (\xi_{t j} \xi_{t j}^T +\alpha I_n ) x_t + x_t^T {\textbf{e}} + U_{t j},\\
4(x_t - {\textbf{e}})^T (x_t-{\textbf{e}}) \leq \Psi_{t j},\\
x_t^T (\xi_{t j} \xi_{t j}^T +\alpha I_n )x_t + x_t^T \xi_{t j} + 1  \leq \Psi_{t j},\\
-100\,{\textbf{e}} \leq x_t \leq 100\,{\textbf{e}},\\
q \geq \theta_{t+1}^i + \langle \beta_{t+1}^i , x_t \rangle - \eta_{t+1}^i,\;i=0,\ldots,k,
\end{array}
\end{equation}
where $\lambda_{1 j}^*$ is an approximate 
value for the optimal Lagrange multiplier associated to
the first constraint (any approximate
primal-dual solution can be used, for
instance running a few iterations
of a quadratic solver).
We then define the Lagrangian 
$
L(f,q,x_t, x_{t-1}, \lambda_1,\xi_t)
=f+q+\lambda_{1}((x_t-x_{t-1})^T (\xi_t \xi_t^T +\alpha I_n )(x_t-x_{t-1}) + x_t^T \xi_t + 1-f)
$
obtained dualizing the
coupling constraint and compute
for $j=1,\ldots,N_t$,
the optimal value 
$\eta_{tj}^k$ of
$$
\begin{array}{l}
\displaystyle \min_{f,q,x_t} \;(1-\lambda_{1 j}^*)(f-f_{t j}^*)+ \langle 
\lambda_{1 j}^*(\xi_{t j}+2(\xi_{t j} \xi_{t j}^T+\alpha I_n)(x_{t j}^*-x_{t-1}^k)),x_t-x_{t j}^* \rangle +  q - q_{t j}^*\\
{\bar f}_{t j} \geq f \geq x_t^T (\xi_{t j}\xi_{t j}^T +\alpha I_n ) x_t + x_t^T {\textbf{e}} + U_{t j},\\
4(x_t - {\textbf{e}})^T (x_t-{\textbf{e}}) \leq \Psi_{t j},\\
x_t^T (\xi_{t j} \xi_{t j}^T +\alpha I_n )x_t + x_t^T \xi_{t j} + 1  \leq \Psi_{t j},\\
-100\,{\textbf{e}} \leq x_t \leq 100\,{\textbf{e}},\\
q \geq \theta_{t+1}^i + \langle \beta_{t+1}^i , x_t \rangle
-\eta_{t+1}^i,\;i=0,\ldots,k,\\
\end{array}
$$
where ${\bar  f}_{t j}$ is an
upper bound for $f_t(\cdot,\cdot,\xi_{t j})$
on $\mathcal{X}_t \times \mathcal{X}_{t-1}:=
\{x_t: -100\,{\textbf{e}} \leq x_t \leq 100\,{\textbf{e}}\}
\times \{x_{t-1}: -100\,{\textbf{e}} \leq x_{t-1} \leq 100\,{\textbf{e}}\}.
$ 
Setting
$
\beta_{t j}^k = 
2 \lambda_{1 j}^{*}(\xi_{t j}\xi_{t j}^T + \alpha I_n)(x_{t-1}^k-x_{t j}^*)
$
and
$$
\theta_{t j}^k = 
L(f_{t j}^*,q_{t j}^*,x_{t j}^*, x_{t-1}^k, \lambda_{1 j}^*,\xi_{t j})- \langle \beta_{t j}^k ,
x_{t-1}^k \rangle,
$$
the coefficients
$\theta_t^k, \eta_t^k, \beta_t^k$
of the cut $\mathcal{C}_t^k$
are given by 
$$
\theta_t^k = \sum_{j=1}^{N_t} p_{t j} \theta_{t j}^k,\;
\beta_t^k = \sum_{j=1}^{N_t} p_{t j} \beta_{t j}^k,\;\mbox{and }
\eta_t^k = \sum_{j=1}^{N_t} p_{t j} \eta_{t j}^k.
$$
If instead of approximate primal-dual solutions we compute exact primal-dual solutions, we get $\eta_{t j}^k=0$,
$L(f_{t j}^*,q_{t j}^*,x_{t j}^*, x_{t-1}^k, \lambda_{1 j}^*,\xi_{t j})=f_{t j}^*+
q_{t j}^*$ and we get the usual cut
computed by SDDP applied to convex problems.

\addcontentsline{toc}{section}{References}
\bibliographystyle{plain}
\bibliography{SIAM_Jan_2020}

\end{document}


\maketitle

\section{A detailed example}

Here we include some equations and theorem-like environments to show
how these are labeled in a supplement and can be referenced from the
main text.
Consider the following equation:
\begin{equation}
  \label{eq:suppa}
  a^2 + b^2 = c^2.
\end{equation}
You can also reference equations such as \cref{eq:matrices,eq:bb} 
from the main article in this supplement.

\lipsum[100-101]

\begin{theorem}
  An example theorem.
\end{theorem}

\lipsum[102]
 
\begin{lemma}
  An example lemma.
\end{lemma}

\lipsum[103-105]

Here is an example citation: \cite{KoMa14}.

\section[Proof of Thm]{Proof of \cref{thm:bigthm}}
\label{sec:proof}

\lipsum[106-114]

\section{Additional experimental results}
\Cref{tab:foo} shows additional
supporting evidence. 

\begin{table}[htbp]
  \caption{Example table}
  \label{tab:foo}
  \centering
  \begin{tabular}{|c|c|c|} \hline
   Species & \bf Mean & \bf Std.~Dev. \\ \hline
    1 & 3.4 & 1.2 \\
    2 & 5.4 & 0.6 \\ \hline
  \end{tabular}
\end{table}

\bibliographystyle{siamplain}
\bibliography{references}